\documentclass[11pt, a4paper]{amsart}
\usepackage{amscd,amssymb,eucal,mathrsfs,amsmath}
\usepackage{hyperref} 

\input xy
\xyoption{all}
\usepackage{epsfig}

\newtheorem{thm}{Theorem}[section]
\newtheorem{cor}[thm]{Corollary}
\newtheorem{lem}[thm]{Lemma}
\newtheorem{prop}[thm]{Proposition}

  \newtheorem{f}[thm]{Fact}                              
\theoremstyle{definition}
\newtheorem{defin}[thm]{Definition}

\theoremstyle{remark}
\newtheorem{remark}[thm]{Remark}

\newtheorem{example}[thm]{Example}

\numberwithin{equation}{section}

\newcommand{\delete}[1]{} 
\newcommand{\nt}{\noindent}

\def\eps{{\varepsilon}}

\def\s{\sigma}

\def\a{\alpha}

\def\eps{{\varepsilon}}

\newcommand{\g}{\gamma}

\newcommand{\cU}{\mathfrak{U}}

\def\~dl{\overline{ \delta }}

\newcommand{\sk}{\vskip 0.2cm}

\newcommand{\ben}{\begin{enumerate}}

\newcommand{\een}{\end{enumerate}}
\newcommand{\bit}{\begin{itemize}}

\newcommand{\eit}{\end{itemize}}

\def\R {{\mathbb R}}
\def\N {{\mathbb N}}
\def\Z {{\mathbb Z}}
\def\Q {{\mathbb Q}}

\def\U{\mathcal{U}} 

\def\T {{\mathbb T}}
\def\F {{\mathbb F}}

\def\RUC{{\hbox{RUC}}}

\def\Iso{{\mathrm{Iso}}}

\def\Homeo{{\mathrm{Homeo}}}

\newcommand{\TFAE}{The following are equivalent: }

\def\QED{\nobreak\quad\ifmmode\roman{Q.E.D.}\else{\rm Q.E.D.}\fi}

\begin{document}

\title[]{Non-archimedean topological monoids}

		\author[]{M. Megrelishvili}
\address[M. Megrelishvili]
{\hfill\break Department of Mathematics
	\hfill\break
	Bar-Ilan University, 52900 Ramat-Gan	
	\hfill\break
	Israel}
\email{megereli@math.biu.ac.il}

\author[]{M. Shlossberg}
\address[M. Shlossberg]
{\hfill\break School of Computer Science
	\hfill\break Reichman University, 4610101 Herzliya 
	\hfill\break Israel}
\email{menachem.shlossberg@post.runi.ac.il}


\date{2024, June}


\thanks{{\it 2020 AMS classification:} 54H15, 26E30, 54D35, 18F70} 
\thanks{This research was supported by a grant of the Israel Science Foundation (ISF 1194/19).} 

\begin{abstract}   
We say that a topological 
monoid $S$ is left non-archimedean (in short: l-NA) if the left action of $S$ on itself admits a proper $S$-compactification $\nu \colon S \hookrightarrow Y$ such that $Y$ is a Stone space. This provides a natural generalization of the well known concept of NA topological groups. The Stone and Pontryagin dualities play major role in achieving useful characterizations of NA monoids. 
We show that many naturally defined topological monoids are NA and 
present universal NA monoids. 
Among others, we prove that the Polish monoid 
 $C(2^{\omega},2^{\omega})$ is a universal separable metrizable 
l-NA monoid and the Polish monoid $\N^{\N}$ is universal for separable metrizable r-NA monoids. 
\end{abstract}

\maketitle  
\setcounter{tocdepth}{1}
 \tableofcontents
    
\section{Introduction}
 
A topological group $G$ is non-archimedean (NA, in short) if 
it has a local basis at the identity every member of which is an open subgroup of $G$. 
The importance of NA topological groups is well known in topology and non-archimedean analysis. They play a central role in the Kechris-Pestov-Todorcevic \cite{KPT} theory regarding Fra\"{i}ss\'{e} structures. For example, recall a characterization of Polish NA groups as the automorphism groups $\rm{Aut}(\mathbb{A})$ of countable Fra\"{i}ss\'{e} structures $\mathbb{A}$, \cite[Section 6.6]{Pe-nbook}.  

There are several equivalent definitions for NA groups. 
One may show (see \cite{MS1}) that $G$ is NA if and only if it admits a proper $G$-compactification $\nu \colon G \hookrightarrow Y$, where $Y$ is a Stone space (i.e., $Y$ is compact and zero-dimensional) and $G$ is treated as a $G$-space with respect to the usual left action.  

This reformulation suggests a natural analog for topological monoids.  
Synthesizing some ideas and techniques from the papers \cite{Me-cs07} (about compactifiable monoids) and \cite{MS1} (about NA groups), 
we introduce 
{\it non-archimedean monoids}. 
More precisely, we say that a topological monoid $S$ is \textit{left non-archimedean} 
 (in short, l-NA) if there exists a  proper $S$-compactification $\nu \colon S \hookrightarrow Y$ of the left action of $S$ on itself, where $Y$ is a Stone space (Definition \ref{d:NA-action}). 
 Similarly we define the \textit{right non-archimedean monoids} (r-NA).  

 Note that a topological monoid $S$ is l-NA if and only if its opposite monoid $S^{op}$ is r-NA. A topological group $G$  is l-NA if and only if it is r-NA as the inversion map $G \to G^{op}, \ g \mapsto g^{-1}$ is a topological isomorphism.
 In contrast to  NA groups, there is a clear asymmetry for monoids. There are l-NA monoids which are not r-NA and vice versa (see Example \ref{ex:contrast}). 
 
Every NA monoid is zero-dimensional (that is, has a basis which consist of clopen subsets). Not every zero-dimensional (locally compact second countable) monoid is NA. See Example \ref{ex:Q} and Remark \ref{r:notNA}.  

In \cite{MeShNAtransp}, we introduced non-archimedean transportation problems via the naturally arising  Kantorovich ultra-norms. 
In Proposition \ref{p:NAvalued}, we use these ultra-norms 
and NA Arens-Eells type theorem \cite[Theorem 4.2]{MeShNAtransp}, 
to prove that every  r-NA monoid  is a topological submonoid of $\Theta_{lin}(V)$, the set of all nonexpansive  linear operators on $V$,  where $V$  is an ultra-normed $\F$-vector space and $\F$  is an arbitrary NA valued field. 

The following characterization theorem,  
which we prove in Section \ref{s:dualities},  
demonstrates that the class of all NA monoids is very large and contains many important examples.  

 \begin{thm} \label{t:conditSEM} 
 	The following assertions are equivalent for every topological monoid $S$.  
 	
 	\ben
 \item $S$ is an l-NA topological monoid.
  
 \item $S$ is a topological submonoid of $C(Y,Y)$ for some Stone space $Y$  (where $w(Y) = w(S)$).
 
 \item 
 The opposite monoid $S^{op}$ can be embedded into 
 the monoid $\rm{End}_{R}(B)$ of endomorphisms of some discrete Boolean ring $B$ (with cardinality $|B| \leq w(S)$). 
 
 \item
 $S^{op}$ is a topological submonoid of $D^{D}$
 for some discrete set $D$ (where $|D|\leq w(S)$).
 
 \item There exists an ultra-metric space $(M,d)$ such that 
 $S^{op}$ is 
 a topological submonoid of the monoid $\Theta(M,d)$  of all 1-Lipschitz maps $M \to M$ equipped with the pointwise topology (where $w(M) \leq w(S)$). 
 
 \item There exists a topologically compatible uniformity $\U$ on $S$ which is generated by a family of right $S$-nonexpansive  ultra-pseudometrics. 
 
 \item
 $S$ is topologically isomorphic to a submonoid of $\rm{Unif}(Y,Y)$ for some NA uniform space $(Y,\mathcal V)$.

 \item 	$S$ can be embedded into the monoid $\rm{End}(K)$ of endomorphisms of some profinite Boolean group $K$  (with $w(K) \leq w(S)$).  
 \item $S$ can be embedded into the monoid $\rm{End}(K)$ of endomorphisms for a compact
 abelian topological group $K$ (with $w(K) \leq w(S)$).

 	%
 \een
 \end{thm}

 One may expect that several naturally defined monoids in NA functional analysis are NA (see, among others, Proposition \ref{p:NAvalued}).   
 Among the technical tools we use in the present paper are the Stone and Pontryagin dualities (Section \ref{s:dualities} and, in particular, Theorem  \ref{p:anti-iso}). Also the factorization Theorem \ref{t:G-SklyarenkoSEMmetr} for monoid actions provides an important technical tool. 
 

 Interesting additional source of NA monoids is the left (or, right) completion of NA groups. According to Proposition \ref{p:r-compl}  if $G$ is an NA group, then 
 the left completion $\widehat{G}^l$ is an r-NA monoid and the right completion $\widehat{G}^r$ is an l-NA monoid.  
 Recall that the topological monoids $\widehat{G}^l$ are important objects in the K-P-T theory \cite{KPT}; namely, they provide a useful tool understanding the \textit{oscillation stability}. We refer to  \cite{Pe-nbook} and \cite{KPT} for more details.

 As usual, by the symmetric group $S_D$ we mean the group of all permutations of a set $D$ with the pointwise topology (inherited from $D^D$). 
 Recall that $S_D$ is universal for NA groups with weight $w(G) \leq |D|$ (see, for example, \cite{MS1}). 
 For example, the Polish group $S_{\N}$ is universal for all second countable (Polish) NA groups. The same is true for the Polish group $\Homeo(2^{\omega})$ (homeomorphism group of the Cantor cube $2^{\omega}$), \cite{MeSc01}. 
 
 We show that similar results hold for NA monoids. 
  More precisely, we prove the following results in Theorem \ref{c:universal} and Theorem \ref{c:CANTOR}.

 \begin{thm} 
 	(Universal NA monoids)  
 	\begin{enumerate}
 		\item 	The Polish monoid 
 		$\N^{\N}$ is a universal separable metrizable 
 		r-NA monoid. 
 		More generally, $\kappa^{\kappa}$ is a universal r-NA monoid of weight $\kappa$ for every infinite cardinal $\kappa$. 
 		\item 	
 		 The Polish monoid $C(2^{\omega},2^{\omega})$ is universal for separable metrizable l-NA monoids. 
 	\end{enumerate} 
 \end{thm}
 
 Moreover, the action of $C(2^{\omega},2^{\omega})$ on $2^{\omega}$ is  
 universal in the class of all actions $S \times Y \to Y$, where $Y$ is a  metrizable Stone space and $S$ is a topological submonoid of $C(Y,Y)$.  
 
 According to Proposition \ref{p:emb},  
 $C(2^{\omega},2^{\omega})$ is embedded into $(\N^{\N})^{op}$ and $(\N^{\N})^{op}$ is embedded into $C(2^{\omega},2^{\omega})$.  
 On the other hand, $(\N^{\N})^{op}$ and $C(2^{\omega},2^{\omega})$ are not isomorphic as topological monoids. 
 
 Recall also a known result from \cite{Me-cs07} which asserts that the Polish topological monoid  $C([0,1]^{\omega}, [0,1]^{\omega})$ is universal for separable metrizable (hence, also Polish) left compactifiable monoids and the action of $C([0,1]^{\omega}, [0,1]^{\omega})$ on $[0,1]^{\omega}$ is a universal left compactifiable action. 
According to an earlier result of Uspenskij \cite{Usp-86}, the topological group  $\Homeo([0,1]^{\omega})$ is universal for all Polish topological groups.  
 
 \sk 
 \noindent \textbf{Acknowledgment}. 
  We thank Vladimir Pestov for his valuable insights. We are very grateful to the  
  referees who suggested several important improvements.
 
\sk  
\section{Preliminaries and some examples of topological monoids} 

All topological spaces below are  usually  assumed to be  
Tychonoff.  
Following \cite{Isb}, a uniformity that is not necessarily Hausdorff is called a \textit{pre-uniformity}.  
We say that a (Tychonoff) topological space $X$ is zero-dimensional if it has a  topological basis which consists of clopen subsets. 
This property is hereditary. If $X$ is compact then it is zero-dimensional if and only if its covering dimension $\dim (X)$ is zero. That is, every finite open covering has a finite open refinement which is a partition of $X$.  
A \textit{Stone space} is a compact zero-dimensional space. 
A \textit{Polish space} is  a topological space homeomorphic to a complete metric space

As in \cite{MS1}, 
we say that a uniformity on a set $X$ is NA if it is zero-dimensional ($\dim \U=0$). It is equivalent to say that there exists a uniform base $\g$ of $\U$ such that every entourage $\eps \in \g$ is an equivalence relation on $X$.     
Every compact space $K$ with its unique compatible uniformity is NA if and only if $K$ is a Stone space.   

A dense continuous function $\nu \colon X \to Y$ into a Hausdorff space is a compactification of $X$. We say that $\nu$ is \textit{proper} if it is a topological embedding.  

Recall that the Samuel compactification $u \colon (X,\U) \to uX$ for a Hausdorff uniformity $\U$ is the (always proper) compactification induced by the algebra $\rm{Unif}_b(X,\R)$ of all bounded $\U$-uniformly continuous functions.  

\begin{lem} \label{l:0} 
	Let $(X,\U)$ be a zero-dimensional uniform space. Then the corresponding Samuel compactification 
	$\nu_s \colon X \hookrightarrow uX$ is a proper zero-dimensional compactification. 
	It can be identified with completion of the precompact replica $\U^*$. Open equivalence relations $\eps \in \U$ on $X$ with finitely many equivalence classes form a uniform basis $B$ of $\U^*$.  
	In particular, $uX$ is a Stone space. 
\end{lem}
\begin{proof}
	See for example \cite[Ch. 5]{Isb} for a much more strong result. 
\end{proof}

For simplicity, we consider topological monoids (instead of semigroups) and monoidal actions. 
By an action of a monoid $S$ on a set $X$ we mean a left or right monoidal action. 
That is, $(st)(x)=s(t(x))$ for every $s,t \in S, \ x \in X$ and  the identity element $e_S \in S$ acts as the identity transformation of $X$.
Speaking about homomorphisms (in particular, embeddings) of monoids we always assume that the homomorphisms preserve the identity. 

It is worth noting that our results can be extended to topological semigroups which are not necessarily monoids. The reason is that 
 	every topological semigroup $S$ can be canonically  embedded into the
 	topological monoid $S_e:=S \sqcup \{e\}$ as a clopen subsemigroup
 	by adjoining to $S$ an isolated identity $e$. Furthermore, any
 	action $\pi \colon S \times X \to X$ naturally extended to the monoidal
 	action $\pi_e \colon S_e \times X \to X$ 
 (see \cite[Remark 3.11]{Me-cs07}).

 
\sk 
Recall some natural constructions of topological monoids and monoidal actions. 
We will use them in the sequel to build NA monoids.

\sk  
\subsection{Pointwise topology} 
\label{d:p} 

	Let $Y$ be a topological space and $X$ be a set. 
	Denote by $Y^X$ the set of all maps $f\colon X \to Y$. 
	The \textit{pointwise topology} 
	$\tau_{p}$ on $Y^X$ is the topology having as the topological subbase all sets of the form 
	$$
	[x_0,O]:=\{f \in C(X,Y): f(x_0) \subset O\}. 
	$$ 
	where $x_0 \in X$ and $O$ is open in $Y$. 
	It is just the product topology on $Y^X$.

If $(Y,\U)$ is a uniform space, then $Y^X$ carries the \textit{pointwise uniformity} $\U_p$ which induces the pointwise topology. That is, $top(\U_p)=\tau_p$. Recall that   the following system
of entourages 
$
\{[x_0,\eps]: x_0 \in X, \eps \in \U\}, \ \text{where} \ [x_0,\eps]:=\{(f_1,f_2) \in Y^X: (f_1(x_0),f_2(x_0)) \in \eps)\},
$ is a uniform subbase of $\U_p$.

\sk  
	For every metric space $(M,d)$ denote by $\Theta(M,d)$ the monoid of all 
	{\it 1-Lipschitz maps} $f\colon X\to X$ (that is, $d(f(x),f(y)) 
	\leq d(x,y)$). Then 
	
	\begin{itemize}
		\item [(a)] $\Theta(M,d)$ is a topological monoid with respect to the pointwise topology. 
		\item [(b)] The subset $\Iso (M) \subset \Theta(M,d)$ of all onto isometries is a topological group. 
		\item [(c)] The evaluation map
		$\Theta(M,d) \times M \to M$ is a continuous action.
		
	\nt 	
Note that the pointwise topology $\tau_p$ on $\Theta(M,d)$ is the minimal   topology which guarantees the continuity of the action $\Theta(M,d) \times M \to M$, or, equivalently, of the orbit maps $\tilde{m}\colon \Theta(M,d) \to M, \ s \mapsto s(m)$ for all $m \in M$.  
	\end{itemize} 

\begin{example} \label{ex:D^D} 
	$(D^D, \circ, \tau_p)$ is a topological monoid for every \textit{discrete} set $D$. Indeed, this monoid can be identified with $\Theta(D,d_{\Delta})$, where $d_{\Delta}(x,y)=1$ for every distinct $x, y \in D$. 
	The symmetric (topological) group $S_D$ can be identified with 
	$\Iso (D,d_{\Delta})$.  
\end{example}

\sk 
An action $S \times X \to X$ on a metric space $(X,d)$ is
\emph{nonexpansive} (or, 1-Lipshitz)  
if every $s$-translation $\tilde{s}\colon X \to X$ belongs to 
$\Theta(X,d)$. It defines a natural homomorphism $h\colon S \to \Theta(X,d)$ which is continuous if and only if the action is continuous.

\sk 
Let $(V,||\cdot||)$ be a normed space (over the field $\R$ of reals). 
Denote by $\Theta_{lin}(V)$ the set of all nonexpansive  \textit{linear} operators
on $V.$ That is, 
$$\Theta_{lin}(V)=\{\sigma \in L(V):  ||\sigma|| \leq 1 \}.$$ 
Let $\tau_{sop}$ be the \textit{strong operator topology} on $\Theta_{lin}(V)$. It is the pointwise topology inherited from $(V,d_{||\cdot||})^V$.  
Denote by $\Iso_{lin} (V)$ the group of all onto linear isometries $V \to V$. It is just the group of all invertible elements in the monoid $\Theta_{lin}(V)$. Note that 
	$(\Theta_{lin}(V),\tau_{sop})$ is a topological submonoid of $(\Theta(V,d_{||\cdot||}), \tau_p)$ 
	and plays a major role in analysis.   	

\sk
\subsection{Compact-open topology} 
\label{d:co} 


	Let 
	$X$ and $Y$ be topological spaces 
	and $C(X,Y)$ be the set of all continuous functions $f\colon X \to Y$. 
	The \textit{compact-open topology} 
	$\tau_c$ on $C(X,Y)$ is the topology having as 
	a subbase  all sets of the form 
	$$
	[K,O]:=\{f \in C(X,Y): f(K) \subset O\}, 
	$$ 
	where $K$ is a compact subset of $X$ and $O$ is open in $Y$. 
	
	If $(Y,\U)$ is a uniform space then the uniformity $\U_c$ of compact convergence on $C(X,Y)$ is generated by the following uniform subbase  
	$$
	\{[K,\eps]\}, \ \text{where} \ [K,\eps]:=\{(f_1,f_2) \in C(X,Y) \times C(X,Y): (f_1(x),f_2(x)) \in \eps \forall x \in K)\},
	$$
	where $K$ is a compact subset of $X$ and $\eps \in \U$. Then $top(\U_c)=\tau_c$. 

	Let $Y$ be a compact space. Then the following hold:
	\ben
	\item
	The monoid $C(Y,Y)$ endowed with the compact-open topology is a topological monoid;
	\item
	The subset $\Homeo (Y)$ 
	of all
	homeomorphisms $Y \to Y$ is a topological group;
	\item
	For every submonoid $S \subset C(Y,Y)$ the induced action $S \times Y \to Y$
	is continuous. 
	\nt 
	Furthermore, it satisfies the following {\it minimality property}.
	If $\tau$  is an arbitrary topology on $S$ such that
	$(S,\tau) \times Y \to Y$ is continuous then  $\tau_{c} \subseteq \tau$.
	\een

\sk 
	Let $X$ be a locally compact group and 
	$G=\rm{Aut}(X)$ be the group of all topological automorphisms endowed with the Birkhoff topology (see \cite{HR} or 
	\cite[p. 260]{DPS}).  
	This is the (Hausdorff) group topology $\tau_B$ having as a local base at the identity $id_X$ 
	the sets of the form 
	$$U_{K,O}:=\{f \in \rm{Aut} (X): f(x) \in Ox \ \text{and} \ f^{-1}(x) \in Ox \ \forall x \in K\}$$
	where $K$ is a compact subset in $X$ and $O$ is a neighborhood of $e$ in $X$. 
	Then 
	\begin{itemize}
		\item [(a)] $\rm{Aut}(X)$ is a topological group; 
		\item [(b)] The evaluation map
		$\rm{Aut}(X) \times X \to X$ is a continuous action; 

	\nt Moreover, the following remarkable minimality property holds. 
	For every Hausdorff group topology $\tau$ on $\rm{Aut}(X)$ which satisfies (b) we have $\tau_B \subseteq \tau$.  
	\end{itemize} 

\subsection{Uniformity of uniform convergence} 
\label{d:unif-conv}

	If $(Y,\U)$ is a uniform space, then 
	for every topological space $X$ 
	the uniformity $\U_{sup}$ of 
	\textit{uniform convergence} on $C(X,Y)$ is generated by the following 
	base 
	$$
	\{\widetilde{\eps}: \ \eps \in \U\}, \ \text{where} \ \widetilde{\eps}:=\{(f_1,f_2) \in C(X,Y) \times C(X,Y): (f_1(x),f_2(x)) \in \eps \ \ \forall x \in X)\}. 
	$$ 

\sk 
Let $\{\rho_i: i \in \Gamma\}$ be a system of bounded pseudometrics which generates $\U$. Then 
the system $\{\rho^*_i: i \in \Gamma\}$ generates $\U_{sup}$ where 
$$
\rho^*_i(f_1,f_2):=\sup_{x \in X} \{\rho_i(f_1(x),f_2(x))\}. 
$$   
If $X$ is compact then $\U_{sup}=\U_c$.

Denote by $\rm{Unif}(Y,Y)$ the monoid of all uniformly continuous selfmaps.  
Clearly, $\rm{Unif}(Y,Y)$ is a submonoid of $C(Y,Y)$.

%

\sk 
\begin{example} \label{l:Unif}
	Let $(Y,\U)$ be a uniform space. Then
		$\rm{Unif}(Y,Y)$ with the topology inherited from 
		$top(\U_{\sup})$   is a topological monoid. 
	For every submonoid $S
		\subset \rm{Unif}(Y,Y)$ the induced action $S \times Y \to Y$
		is continuous. 
\end{example}


\begin{remark} \label{r:ResOfElliott}  
	Recall two results from a recent paper by L. Elliott, J. Jonusas, Z. Mesyan, J.D. Mitchell, M. Morayne and Y. Peresse. 
	\begin{enumerate}
		\item \cite[Corollary 6.12]{Elliott} The compact-open topology is the unique second countable Hausdorff monoid topology on the monoid $C(2^{\omega},2^{\omega})$ of continuous functions on the Cantor set $2^{\omega}$.
		\item \cite[Theorem 5.4(b)]{Elliott} The natural pointwise topology is the unique
		Polish monoid topology for $\N^{\N}$.  
	\end{enumerate} 
	
	These two examples appear below (Theorems \ref{c:CANTOR} and \ref{c:universal}) in the context of universal NA monoids.  
\end{remark}

\begin{remark} \label{r:monoids} 
	For every topological group $G$ denote by $\rm{End}(G)$ the monoid  of all continuous endomorphisms 
	under the usual composition.
If $G$ is compact, then $\rm{End}(G)$ is a topological monoid with respect to the 
	compact-open topology, which, by the compactness of $G$, is the topology of uniform convergence. 
	
	If $G$ is discrete, then we get on $\rm{End}(G)$ the pointwise topology. 
	This is a topological submonoid of $(G^G,\circ,\tau_p)$.  
	If $G$ is a discrete ring, then $\rm{End}_{R}(G) \subset \rm{End}(G)$ is the submonoid of all \textit{ring endomorphisms} $G \to G$. 
\end{remark}

\section{Compactifiability of monoid actions} 
\label{s:compactifiability} 

An $S$-\emph{space} is a continuous action of a topological monoid 
$S$ on a topological space $X$. 

Compactifiability of 
topological spaces 
means the existence of topological embeddings into compact
(Hausdorff) spaces. For the compactifiability of $S$-spaces we require, in addition,
 that the original action admits a continuous extension.  
 
\begin{defin} \label{d:S-compactif} 
	Let $X$ be an $S$-space with respect to the continuous left action $S \times X \to X$. 
	\begin{enumerate}
		\item An $S$-compactification of $X$ is a continuous dense $S$-map $\nu \colon X \to K$, where $K$ is a compact (Hausdorff) $S$-space. 
		\item We say that $\nu$ is \textit{proper} if $\nu$ is a topological embedding. 
		\item $X$ is $S$-compactifiable if there exists a proper $S$-compactification of $X$. 
	\end{enumerate} 
\end{defin}

Compactifiable $S$-spaces are known also as $S$-\emph{Tychonoff} \emph{spaces}. The obvious reason of this name is that every $S$-compactifiable $X$ is Tychonoff.  For locally compact topological groups $G$ every Tychonoff $G$-space is $G$-Tychonoff by a celebrated result by J. de Vries \cite{Vr-loccom78}. 
However, not every Tychonoff $S$-space is $S$-Tychonoff even for topological groups $S=G$ (where $G$ and $X$ are Polish); see \cite{Me-Ex88} and \cite{Pest-Smirnov}. 

	Of course, the $S$-compactifications and $S$-compactifiability can be defined also for right actions $X \times S \to X$. 
	In particular, for the particular case of the right action $S \times S \to S$.  
	In this case, we typically warn the reader by adding the word "right". 

\sk 
By a classical result of R. Brook \cite{Br}, 
the left action of a topological group $G$ on itself is compactifiable.  
This fact was the trigger for introducing the \textit{compactifiable monoids} in \cite{Me-cs07}. 

\begin{defin} \label{d:comp}  
	A topological monoid $S$ is said to be 
	\begin{enumerate}
		\item 	\textit{left compactifiable} (or, simply, l-compactifiable) if the left action of $S$ on itself is $S$-compactifiable; 
		\item \textit{right compactifiable} (or r-compactifiable) if the right action of $S$ on itself is $S$-compactifiable. 
	\end{enumerate} 
\end{defin}

Note that  $S$ is an l-compactifiable if and only if $S^{op}$ (the opposite monoid) is r-compactifiable. 

	In view of some results from \cite{MS1}, two remarks are in order here. 
	\begin{itemize}
		\item In contrast to the topological group case, not every Tychonoff monoid $S$ is compactifiable.  Even if $S$ is locally compact and metrizable (see Remark \ref{r:notNA}).  
		\item The asymmetry between left and right compactifiability for topological monoids is also remarkable. There are topological monoids which are left but not right compactifiable (and vice versa). See 
		Example \ref{ex:contrast} below. 
	\end{itemize}


\sk 

\begin{defin} \label{d:equiun}
	Let 
	$\pi \colon S \times X \to X$ be an action and 
	$\U$ be a compatible uniformity on a topological space $X.$ We call the 
	action (sometimes, also $X$) :
	\ben
	\item
	{\it $\U$-saturated} if every $s$-translation ${\tilde s}\colon X \to X$
	is $\U$-uniform. 
	That is, if $s^{-1} \eps \in \U$ \ \ \ $\forall (s,\eps)  \in S \times \U$, where $s^{-1}\eps:=\{(x,y) \in X\times X: (sx,sy) \in \varepsilon\}.$ Equivalently,
	$$
	\forall \eps \in \U \ \  \forall s \in S \ \ \exists \delta \in \U \ \ \ \ (sx,sy) \in \eps  \ \ \forall (x,y) \in \delta.
	$$  
	If $\U$ is saturated, then the corresponding homomorphism 
	$$h_{\pi}\colon S \to \rm{Unif}(X,X), \ s \mapsto \tilde{s}$$ is well defined.
	
	\item
	{\it $\U$-bounded at $s_0$} if for every $\varepsilon \in \U$
	there exists a neighborhood $U\in N_{s_0}$ such that $(s_0x,sx) \in
	\varepsilon$ for each $x\in X$ and $s\in U$. If this condition holds
	for every $s_0 \in S$ then we simply say
	that $X$ is
	 \emph{$\U$-bounded}  or that
	$\U$ is a \emph{bounded} uniformity;
	
	\item  \emph{$\U$-equiuniform} 
	if it is $\U$-saturated and $\U$-bounded. 
	Sometimes we say also that $\U$ is an {\it equiuniformity}. 
	It is equivalent to say that the corresponding homomorphism
	$h_{\pi}\colon S \to \rm{Unif}(X,X)$ is continuous. 
	 By the "$3$-epsilon argument" $\U$-equiuniform action can be expressed 
	in an equivalent form as follows:  
	
	\sk 
	\nt	for every $\varepsilon \in \U$
	there exist a neighborhood $V \in N_{s_0}$ and $\delta \in \U$ such that $(s_1x,s_2y) \in
	\varepsilon$ for each $(x,y) \in \delta$ and $s_1, s_2 \in V$. 

\item The definitions above make sense also for pre-uniformities. 
	\een
\end{defin}

The definition of $\U$-bounded actions appears in \cite{Br} for group actions  (under the name `\emph{motion equicontinuous}') and is very effective in the theory of $S$-compactifications.  It was widely explored in the papers of J. de Vries \cite{Vr-Embed77, Vr-loccom78, Vr-can75}. See also \cite{Me-cs07,GM-prox10,Me-b}.

\textbf{Notation:} For a given $S$ 
denote by $\rm{EUnif}^S$ 
the triples $(X,\U,\pi)$, where  
$(X,\U)$ is a Hausdorff uniform space, $\pi\colon S \times X \to X$ is a (continuous) equiuniform action. 
The class $\rm{EUnif}^S$ is closed under products and subspaces. We use the same notation for right actions. 


\begin{lem} \label{e:easy} \ \cite{Me-cs07,GM-prox10} (see also \cite{Me-b})  
	\ben
	\item  
	Every $\U$-equiuniform action is continuous. 
		\item $\rm{Comp}^S \subset \rm{EUnif}^S$. 
	Every compact $S$-space $X$ is equiuniform (with respect to the unique compatible uniformity on $X$). 
	\item A continuous monoid action $S \times X \to X$ is $S$-compactifiable if and only if it is $\U$-equiuniform with respect to some compatible uniformity $\U$ on $X$. 
	\item The coset $G$-space $G/H$ 
	is $\U_R$-equiuniform for every topological group $G$ and a subgroup $H$. If $H$ is closed in $G$ then $(G/H,\U_R) \in \rm{EUnif}^G$.

	\item For every uniform space $(X,\U)$ and every submonoid $S$ of the topological monoid 
	$\rm{Unif}(X,X)$ the natural action $S \times X \to X$ is 
	$\U$-equiuniform. 
	\item 
	Let $\pi\colon S\times X \to X$ be a 
	$\U$-equiuniform
	action.  
	Then the induced action on the completion ${\widehat \pi}\colon S\times
	\widehat X \to {\widehat X}$ is ${\widehat{\U}}$-equiuniform. 
	In other terms: $(X,\U) \in \rm{EUnif}^S$ implies that $(\widehat{X},\widehat{\U}) \in \rm{EUnif}^S$. 

	
	\een
\end{lem}

One direction in (3) easily follows from (2). 
The following result explains the second direction in (3) and  is well known for group actions (see for example, \cite{Br, Me-EqComp84}).


\sk 
The following proposition will be used in the proof of Theorem \ref{t:conditSEM}. 

\begin{prop} \label{l:noteasy}
	Let $\U$ be a 
	compatible uniformity on a topological space $X$ and $S$ be a topological monoid and consider the
	monoidal action $S \times X \to X.$ 
	\ben
	\item
	The family $\{s^{-1}\varepsilon: s\in S, \varepsilon
	\in \U\},$
	where
	$$s^{-1}\varepsilon:= \{(x,y) \in X\times X : (sx,sy) \in \varepsilon\},$$ 
	is a subbase of a saturated uniformity   
	$\U_S \supset \U.$  
	\item $\U$ is NA if and only if $\U_S$ is NA.
	\item
	If the action is $\U$-bounded,
	then it is also $\U_S$-equiuniform. 	
	\item 
	If all  $s$-translations $X \to X, x \mapsto s(x)$ are 
	continuous, then $\U_S$ generates the same topology as $\U.$  
	\een
\end{prop}
\begin{proof}
	(1) 
	Clearly,  $\U_S \supset \U$ since $e^{-1}\eps=\eps$ for every $\eps\in \U.$	The equality $t^{-1}(s^{-1}\eps)=(st)^{-1}\eps$ implies that the action is $\U_S$-saturated.
	
	(2) 
	Observe that if $\eps\in \U$ is an equivalence relation, then also $s^{-1}\eps$ is an equivalence relation for every given $s\in S.$
	
	(3) By (1), the action is $\U_S$-saturated.
	So, we only have to show  the boundedness of $\U_S$. It is enough to check it for the elements of the uniform subbase $\{s^{-1}\varepsilon: s\in S, \varepsilon
	\in \U\}$. Let us show the boundedness for $s_{0}^{-1}\eps$ at a given element $t_{0} \in S$. Since the action is $\U$-bounded there exists a neighborhood $U$ of $s_0t_0$ such that 
	$(s_0t_0x,ux) \in \eps$ for every $u \in U$ and every $x \in X$. Since $S$ is a topological monoid we can choose a 
	neighborhood $V$ of $t_0$ such that $s_0V \subset U$. Then $(t_0x,tx) \in s_{0}^{-1}\eps$ for every $t \in V$ and  $x \in X$.  
	
	(4) Let $x\in X$ and $s\in S.$ Since the translations are $\U$-continuous it follows that for every $\eps\in \U$ there exists $\delta\in \U$ such that $\delta(x)\subset (s^{-1}\eps)(x).$ This implies that  $\U_S$ generates the same topology as $\U.$
\end{proof}

\begin{f} \label{t:samuel} 
	\cite{Me-cs07,Me-b}  
	Let $X$ be a (Tychonoff) $S$-space.  
	Assume that $\pi\colon S\times X \to X$ is a 
	$\U$-equiuniform monoid action.
	Then the induced action $\pi_u\colon S\times uX \to uX$ on the Samuel
	compactification
	$uX:=u(X,\U)$ is a proper $S$-compactification of $X$.  
\end{f}

\begin{remark} \label{r:ruc} 
	It is well known 
	(see, for example, \cite{BH97}, \cite{Me-cs07}, \cite{Me-b}) 
	that an $S$-space $X$ is compactifiable if and only if it admits sufficiently many (separating points and closed subsets) real valued bounded functions $f \colon X \to \R$ such that 
	$$
	\forall \eps > 0 \ \forall s_0 \in S \ \exists  U \in N_{s_0} \  \ |f(s_0x)-f(sx)| < \eps \ \ \ \forall x \in X \ \forall s \in U.  
	$$
	Some authors call such functions \textit{right uniformly continuous} (RUC), notation $f \in \RUC (X)$. If $S=G$ is a topological group with the left natural action on itself then $\RUC(G)$ is the usual algebra $\RUC(G)$ of all bounded right uniformly continuous functions on $G$.  	
\end{remark}

\sk 
\begin{defin} \label{d:contr-ps} \ 
	\begin{enumerate}  
		\item A pseudometric $d$
		on a monoid $S$ is {\it right nonexpansive} if $d(xs,ys) \leq 
		d(x,y)$  for every $x,y,s \in S$. Similarly can be defined \textit{left nonexpansive} pseudometric. 
		\item
		A uniform structure $\U$ on a monoid $S$ is \emph{right
			invariant} 
		if for every $\eps \in \U$ there exists $\delta \in \U$ such
		that $\delta \subset \eps$ and $(sx,tx) \in \delta$ for every
		$(s,t) \in \delta$, $x \in S$.
	
	\end{enumerate} 
\end{defin}

Here we provide some natural examples. 
\begin{itemize}
	\item For every topological group $G$ the
	right uniformity $\mathcal{R}(G)$ of $G$ is the \emph{unique} right
	invariant compatible uniformity on $G$, \cite[Lemma 2.2.1]{RD}.
\item
Let $(X,\U)$ be a uniform space and $\U_{\sup}$ be the
corresponding uniformity on $\textrm{Unif}(X,X)$. Assume
that $S$ is a submonoid of $\textrm{Unif}(X,X)$. Then the
subspace uniformity $\U_{\sup}|_S$ on $S$ is right invariant. 
		\item  For every right invariant uniformity $\U$ on $S$ the left action of $S$ on itself is $\U$-bounded (Definition \ref{d:equiun}.2).  
\end{itemize} 

\sk 
The following theorem shows that a topological monoid $S$ is
compactifiable (Definition \ref{d:comp}) if and only if $S$ ``lives in natural monoids".

\begin{f} \label{t:sem}  \cite{Me-cs07} 
	Let $S$ be a topological monoid. The following are equivalent:
	\ben
	\item $S$ is \textbf{left} compactifiable;
	\item $S^{op}$ (the opposite monoid of $S$)
	is a topological submonoid of 
	$\Theta_{lin}(V)$ 
	for some normed
	(equivalently, {\it Banach}) space $V$;
	\item $S^{op}$ is a topological submonoid of $\Theta(M,d)$
for some
	metric space $(M,d)$; 
	\item $S$ is a topological submonoid of $C(K,K)$
	for some compact space $K$;
	\item $S$ is a topological submonoid of $\rm{Unif}(Y,Y)$
	for some uniform space $(Y, \U)$;
	
%
	
	\item The topology of $S$ can be generated by
	a family $\{d_i\}_{i \in I}$ of right nonexpansive pseudometrics on $S$. 
	\een 
\end{f}


So the semigroup $\rm{Unif}(Y,Y)$ is \textbf{left} compactifiable for every uniform space $(Y, \U)$. 
The same holds for $C(K,K)$, where $K$ is a compact space,  
while the monoid $\Theta(X,d)$  is \textbf{right} compactifiable for every
metric space $(X,d).$ 
It follows that $\Theta_{lin}(V)$ is right compactifiable for every normed space $V$.  




	


For more facts about compactifiability of monoid actions we refer to \cite{Me-EqComp84, BH97, GM-prox10, Me-b}.

\sk  
\section{Non-archimedean actions and monoids} 

\subsection{Non-archimedean topological groups} 

A topological group $G$ is non-archimedean if it has a local basis every member of which is a (necessarily clopen) subgroup of $G$. 

\begin{f} \label{t:condit} \cite{MS1} 
	The following assertions are equivalent: 
	\ben
	\item $G$ is a non-archimedean topological group.
	
	\item 
	The left (right) uniformity of $G$ is NA.  
	
	\item There exists a $0$-dimensional proper $G$-compactification 
	$\nu \colon G \hookrightarrow Y$ of 
	the natural left (equivalently, right) action of $G$ on itself.   
	\item $G$ is a topological subgroup of $\Homeo(X)$ for some 
	Stone space $X$. 
	\item $G$ is a topological subgroup of the automorphisms group (with the pointwise topology) $\rm{Aut}_{R}(B)$ for some
	discrete 
	Boolean ring $B$. 
	\item
	$G$ is embedded into the symmetric topological group $S_{\kappa}$. 
	\item
	$G$ is a topological subgroup of the group $\Iso(X,d)$ of all
	isometries of an 
	ultra-metric space $(X,d)$, with the topology of pointwise
	convergence

	\item The right (left) uniformity on $G$ can be generated by a
	system of right (left) invariant 
	ultra-pseudometrics.
	
	\item $G$ is a topological subgroup of the automorphism group $(\rm{Aut}(K),\tau_{co})$ for some compact abelian group $K$. 
	\een
\end{f}

 Some other results on NA groups (including also \textit{free non-archimedean groups}) can be found in \cite{MeSh-FrNA13}).  
In this work we introduce 
and study 
\textit{non-archimedean monoids and 
non-archimedean monoid actions}. The definition is based on Stone compactifications.  

\begin{defin} \label{d:NA-action} 
	Let  $S$ be a topological monoid.  
	\begin{enumerate}
		\item 	Let $\a \colon S \times X \to X$ be a continuous left action of 
		$S$ on a (Tychonoff) space $X$. We say that this \textit{action is non-archimedean} (in short: NA) if there exists a \textbf{proper} $S$-compactification $\nu \colon X \hookrightarrow Y$ of $X$, where $Y$ is a Stone space.  
		Similarly, for right actions. 
		\item We say that $S$ is \textit{left non-archimedean} (in short: l-NA) if the left action of $S$ on itself is NA. Similarly, for right actions (in short: r-NA).  
		\item We say that $S$ is lr-NA if it is both l-NA and r-NA. 
	\end{enumerate} 
\end{defin}

From this definition it immediately follows that compact zero-dimensional topological monoids are NA. 
It is straightforward to see that
any submonoid of l-NA (r-NA) monoid is l-NA (r-NA).  
Also, the topological product of l-NA (r-NA) monoids is l-NA (r-NA). 

   

\begin{remark}  \label{r:oppos} \ 
	
	\begin{enumerate}
		\item We can 
		assume  in Definition \ref{d:NA-action}.1 that 
		$$w(X) \leq w(Y) \leq w(X) \cdot w(S)$$ 
		 as it follows 
		 from  Theorem  \ref{t:G-SklyarenkoSEMmetr}. So, if $w(S) \leq w(X)$ (e.g., if $S$ is second countable), then $w(X) = w(Y)$.  
		\item 	 Any l-NA  (r-NA) action is l-compactifiable (resp., r-compactifiable) and every l-NA (resp., r-NA) monoid is l-compactifiable (resp., r-compactifiable). 
		
		\item $S$ is l-NA if and only if $S^{op}$ (the opposite monoid) is r-NA. 
	\end{enumerate} 
\end{remark}


\begin{example} \label{ex:Q} 
	By Definition \ref{d:NA-action} it follows that every NA monoid $S$ must be zero-dimensional. 
	However, it is easy to present compactifiable zero-dimensional commutative monoids (even abelian groups) which are not NA. 
	 For example, take the group of all rationals $\Q$. 

\end{example}

\sk 
\begin{prop} \label{t:crit} 
	Let $\a \colon S \times X \to X$ be a continuous action. \TFAE
	\begin{enumerate}
		\item The action $\a$ is NA. 
		\item There exists a topologically compatible \textbf{NA uniformity} $\U$ on $X$ such that the action $\a$ is $\U$-equiuniform. 
		\item \rm{RUC} functions (see Remark \ref{r:ruc}) $f \colon X \to \{0,1\}$ 
		separate points and closed subsets. 
	\end{enumerate}
\end{prop}
\begin{proof} 
	(1) $\Rightarrow$ (2) 
	Note that $\rm{Comp}^S \subset \rm{EUnif}^S$ by  \ref{e:easy}.2. 
	
		(2) $\Rightarrow$ (1)  Combine 
		Fact \ref{t:samuel} and Lemma \ref{l:0}. 
	
		(2) $\Rightarrow$ (3) 
		If the action is $\U$-equiuniform, then, in particular, $X$ is $\U$-bounded. This implies that every bounded $\U$-uniformly continuous function $f \colon X \to \R$ is RUC (Remark \ref{r:ruc}). 
Let $A$ be a closed subset of $X$ and $x\notin A$.	Then there exists an 
open 
 equivalence relation $\eps:=\eps_{x,A} \in \U$ such that $\eps(x)\subseteq A^c.$ It is easy to see that the characteristic function 
 $\chi_{\eps(x)} \colon X\to\{0,1\}$
 is a bounded $\U$-uniformly continuous function. This proves that $\{0,1\}$-valued \rm{RUC} functions 
separate points and closed subsets.		

	(3) $\Rightarrow$ (2) The initial uniformity $\mathcal{U}$ with respect to the family $\g:=\{f \colon X \to \{0,1\}\}$ of all  $\{0,1\}$-valued \rm{RUC} functions   
	is an NA compatible uniformity  on $X.$ Indeed, the equivalence relations $$\eps_f:=\{(x,y)\in X\times X| \  f(x)=f(y)\},$$ where $f\colon X\to \{0,1\}$ is \rm{RUC}, form a subbase for $\mathcal{U}$. We will show that  the action $\a$ is $\mathcal{U}$-equiuniform.  
Note that	$f \colon X \to \{0,1\}$ is \rm{RUC} if and only if  
	$$
\forall s_0 \in S \ \exists V \in N_{s_0} \  (s_0x,sx)\in \eps_f  \ \ \forall x \in X \ \forall s \in V.  
	$$ This implies that $X$ is
	$\U$-bounded. It remains to show that $X$ is
	$\U$-saturated. Let $t_0 \in S$. We have to show that the corresponding translation $X \to X, x \mapsto t_0x$ is $\U$-uniform. 
	First observe that for every \rm{RUC} function $f \colon X \to \{0,1\}$ the composition $ft_0$ (defined by $ft_0(x):=f(t_0 x$)) is also \rm{RUC} and $\{0,1\}$-valued. Therefore, $ft_0 \in \g$.  Next, it is easy to show that for every $\eps_f \in \g$ as above we have 
	$$
	(x,y) \in \eps_{ft_0} \Rightarrow (t_0x, t_0y) \in \eps_f. 
	$$
	This implies that $t_0$-translation is $\U$-uniform because 
	$\g$ is a subbase of $\U$. 	
%
%
%
\end{proof}


\sk 
\begin{prop} \label{t:Theta-ultrametric} 
	If $d$ is an 
	ultra-metric
	 on $M,$ then the topological monoid $\Theta(M,d)$ is r-NA. 
\end{prop}
\begin{proof} 
	We use Proposition \ref{t:crit} for the right action $S \times S \to S$, where $S:=(\Theta(M,d), \tau_p)$. 
	
	The standard basis of the pointwise uniformity $\U_p$ on $\Theta(M,d)$ is the family 
	$$\{\eps_A: A \ \text{is a finite subset of} \ D\}$$
	where $\eps >0$ and 
	$$
	\eps_A:=\{(f_1,f_2) \in S \times S: \ d(f_1(a),f_2(a)) < \eps \ \forall a\in A\}. 
	$$
	Since $d$ is a ultra-metric, every $\eps_A$ is an equivalence relation on $S$. Therefore, we obtain that the uniformity $\U_p$ is NA.  
	
	Now we show that the right action $(S,\U_p) \times S \to (S,\U_p)$ of $S$ on itself is $\U_p$-equiuniform. 
	Let $A \subset S$ be a given finite subset of $S$ and $s_0 \in S$. 
	We have to show that there exists $O \in N(s_0)$ such that 
	$$
	(fs_0,fs) \in \eps_A \ \forall s \in O \ \forall f \in S. 
	$$ 
	It is enough to take 
	$$O:=\{s \in S: \ d(s(a),s_0(a)) < \eps \ \ \forall a \in A\}.$$ 
	Indeed, take into account that every $f \in S=\Theta(M,d)$ is a $1$-Lipschitz map. Hence, $d(f(s(a),f(s_0(a)) \leq d(s(a),s_0(a)) < \eps.$
	
	Now we check that the action is saturated. Let $s_0 \in S$. 
	For every given finite $A \subset M$ its image $s_0A$ is also a finite subset of $M$. Then $(f_1s_0,f_2s_0) \in \eps_A$ for every 
	$(f_1,f_2) \in \eps_{s_0A}$. 	 
\end{proof}

\sk 
As we already mentioned in Section \ref{s:compactifiability} the topological monoid $(D^D, \circ, \tau_p)$ is right compactifiable for every \textit{discrete} set $D$.  
The following result says more. 

\begin{cor} \label{c:D^D}
	For every discrete set $D$ the monoid $(D^D, \circ, \tau_p)$ 
	is r-NA. 
\end{cor}
\begin{proof} $(D^D, \circ, \tau_p)$ is a topological submonoid of $\Theta(D,d_{\Delta})$ by Example \ref{ex:D^D}. So, we can apply 
	Theorem \ref{t:Theta-ultrametric}. 
\end{proof}

\begin{cor} \label{c:Emb} 
	For every metric space $(M,d)$ the monoid $\rm{Emb}(M,d)$ of all 
	isometric embeddings $M \hookrightarrow M$ with respect to the pointwise topology is right-compactifiable. If in addition, $d$ is a ultra-metric, then $\rm{Emb}(M,d)$ is r-NA.  
	In particular, the monoid $\rm{Inj}(D)$ of all injections $D \hookrightarrow D$ is r-NA monoid for every discrete set $D$. 
\end{cor}
\begin{proof} 
	$\rm{Emb}(M,d)$ is a natural topological submonoid of $\Theta(M,d)$. Now apply
	 Proposition \ref{t:Theta-ultrametric}. 
\end{proof}

\sk 
\begin{prop} \label{t:UNIF} 
	For every NA uniform space $(Y,\U)$ 
	the topological monoid $\rm{Unif}(Y,Y)$ (from Example \ref{l:Unif})  is l-NA. 
\end{prop}
\begin{proof} 
	 We use Proposition \ref{t:crit} for the 
	left action $S \times X \to X$, where $X:=S=\rm{Unif}(Y,Y)$. 
If $(Y,\U)$ is NA, then also $(C(X,Y),\U_{sup})$ 
(which is defined in Subsection \ref{d:unif-conv}) 
is NA (because if $\eps \in \U$ is an equivalence relation on $Y$ then $\widetilde{\eps}$ is an equivalence relation on $C(X,Y)$).  
In particular, the uniformity $\U_{sup}$ is NA on the topological monoid $\rm{Unif}(Y,Y)$.

	Now we show that the left action of $S:=\rm{Unif}(Y,Y)$ on itself is $\U_{sup}$-equiuniform. Let $s_0 \in S$ and $\widetilde{\eps} \in \U_{sup}$. Choose the neighborhood $\widetilde{\eps}(s_0):=\{s\in S: (s,s_0) \in \widetilde{\eps}\}$. Then $(sf,s_0f) \in \widetilde{\eps}$ for every $f \in S$
	 and $s\in \widetilde{\eps}(s_0).$ Since every $s \in S$ is a uniform map $(Y,\U) \to (Y,\U)$, it is straightforward to see that the left action $S \times (S,\U_{sup}) \to  (S,\U_{sup})$ is saturated. 
\end{proof}

\begin{cor} \label{c:C(K,K)} 
	Let $K$ be a Stone space. Then the topological monoid $C(K,K)$ is l-NA in its uniform topology.
\end{cor}


\sk 
\subsection{Weil completions of NA groups}  

It is well known that the right and left completions $(\widehat{G}^r, \widehat{\mathcal{U}_R)}$ and $(\widehat{G}^l, \widehat{\mathcal{U}_L)}$ of a topological group $G$ are 
naturally defined (opposite to each other) topological monoids (see for example \cite[Proposition
10.12(a)]{RD}) containing $G$ as a (dense) submonoid.  
These natural monoids are not groups in general. Probably, the first who discovered this fact was J. Dieudonne \cite{Dieudonne}.

\begin{f}  \label{f:completions} \cite{Me-cs07} 
	Let $G$ be a topological group. The monoid $\widehat{G}^r$ is l-compactifiable for every topological group $G$. Similarly, the monoid $\widehat{G}^l$ is r-compactifiable.  
\end{f}

\begin{remark}
	In particular, 
	if $G$ is abelian then its completion $\widehat{G}$ is an (abelian) NA topological group. In fact, the following more general result is true. For every topological group $G$ its \textit{Raikov completion} (completion with respect to the two-sided uniformity) $\widehat{G}$ is an NA topological group. Indeed, if $G$ is NA then by Fact \ref{t:condit} (assertion 4) $G$ is a topological subgroup of $\Homeo(X)$ for some Stone space $X$. 
	Recall that $\Homeo(X)$ is Raikov complete for every compact $X$ (see, for example, \cite{Bourb1}).  
	Then the closure $cl(G)$ of $G$ in $\Homeo(X)$ can be identified with the completion $\widehat{G}$. 
\end{remark}

\begin{prop} \label{p:r-compl}  
		For every topological group $G$ the following conditions are equivalent:
		\begin{enumerate}
			\item $G$ is an NA group; 
			\item  $\widehat{G}^r$ is an l-NA monoid; 
			\item $\widehat{G}^l$ is an r-NA monoid.
		\end{enumerate} 
\end{prop}
\begin{proof} 
	(1) $\Rightarrow$ (2) The completion of NA uniform space is again NA (see \cite[Ch. V]{Isb}). Hence, 
	$Y:=(\widehat{G}^r, \widehat{\mathcal{U}_R)}$ is NA. Now observe that 
	$\widehat{G}^r$ is naturally embedded into $\rm{Unif}(Y,Y)$ and apply Proposition \ref{t:UNIF}.  
	 
	 (2) $\Rightarrow$ (1) 
	  $S:=\widehat{G}^r$ is an l-NA monoid means that there exists a proper $S$-compactification $\nu\colon S \hookrightarrow Y$ of the left action of $S$ on itself. Since $G$ is embedded into $S$, in particular, we obtain a proper $G$-compactification of of the left action of $G$ on itself. 
	  By Fact \ref{t:condit} (assertion 3) we conclude that $G$ is NA.  
	
	
	The equivalence (2) $\Leftrightarrow$ (3) can be proved   using  the formula  $(\widehat{G}^l)^{op}=\widehat{G}^r$.
\end{proof}

\begin{remark} \label{r:Pest} 
	There are concrete descriptions of $\widehat{G}^r$ and $\widehat{G}^l$ for several remarkable groups $G$. For instance, 
	\begin{enumerate}
		\item (J. Dieudonne \cite{Dieudonne}) For the symmetric group $G:=S_{\N}$ its left completion $\widehat{G}^l$ can be identified with the topological monoid $\rm{Inj}(\N)$ of all embeddings (injective maps) $\N \hookrightarrow \N$.  
		\item (Pestov \cite[Prop. 8.2.6]{Pe-nbook}) If $(M,d)$ is a complete metric space and $G:=\Iso(M,d)$ then $\widehat{G}^l$ is a natural topological submonoid of $\rm{Emb}(M,d)$ of all  isometric embeddings $M \hookrightarrow M$. If, in addition, $(M,d)$ is ultra-homogeneous 
		(e.g., if $M$ is 
		a Urysohn space), 
		then $\widehat{G}^l=\rm{Emb}(M,d)$.  
		It is a far reaching generalization of (1). 
		
		Note that the (non-archimedean) topological monoids $\widehat{G}^l$ (for NA groups $G$) and $\rm{Emb}(M,d)$ are important objects in K-P-T theory \cite{KPT}; namely, 
		they provide a useful tool for understanding the oscillation stability. We refer to  \cite{Pe-nbook} and \cite{KPT} for 
		more details. 
		
			\item (Pestov \cite[Prop. 8.2.6]{Pe-nbook}) The left completion of $G=\rm{Aut}(\Q,\leq)$ is the monoid of all order-preserving injections $\Q \hookrightarrow \Q$ with the pointwise topology for discrete $\Q$.  
	\end{enumerate} 
\end{remark}

\sk        
\section{Using Stone and Pontryagin dualities} 
\label{s:dualities} 

By Stone's celebrated representation theorem, there is a 
duality between Boolean algebras (or, Boolean rings) and zero-dimensional compact spaces (Stone spaces). More precisely, for every Stone space $Y$ we have the Boolean algebra ($clop(Y), \cup, \cap)$ of all clopen subsets. 
Conversely, for every Boolean algebra $B$ the set of all ultra-filters under a naturally defined topology is a Stone space. Moreover, 
one may 
retrieve the original structure by applying each of these constructions. 

\begin{remark} \label{r:rings} 
	It is also well known that one may equally consider the Boolean rings 
	associated with the Boolean algebras. In this case the original Stone space can be reconstructed as the set of all ring homomorphisms $B \to \Z_2$. 
	We refer to \cite{GH-Stone} for more details about this classical theory. 
\end{remark}

Let $Y$ be a Stone space and 
$B=(B(Y), \triangle, \cap)$ be 
the discrete Boolean ring of all clopen
subsets in $Y$, where 
 symmetric difference and intersection   serve as 
 the addition and multiplication, respectively.
As usual, one may identify $B$ with the Boolean ring  $B:=C(Y,\mathbb{Z}_2)$ of all
continuous functions $\chi \colon Y\rightarrow \mathbb{Z}_2.$
By 
the 
standard compactness arguments, it is clear that $|B|=w(Y)$. 

Denote by $B^{\ast}:=\rm{Hom}(B,\T)$ the Pontryagin dual of $B.$
Since $B$ is a Boolean group (that is, $\chi = -\chi$ for every $\chi \in B$), 
every character $B \to \T$ can be identified with a group homomorphism into the unique
2-element subgroup $\Omega_2=\{1, -1\}$, a copy of $\Z_2$. The same
is true for the characters on $B^*$, hence the natural evaluation
map 
$w \colon B \times B^*  \to  \T$, $w(\chi,f)=f(\chi)$ can be restricted
naturally to $B \times B^* \to \Z_2$. Under this identification
$B^{\ast}:=\rm{Hom}(B,\Z_2)$ is a closed (hence compact) subgroup of the
compact group $\mathbb{Z}_2^{B}.$
In particular, $B^*$ is a Boolean profinite group.

Clearly, the groups $B$ and $\mathbb Z_2$, being discrete, are
non-archimedean. The group $ B^{\ast}=\rm{Hom}(B,\mathbb{Z}_2)$ is also
non-archimedean since it is a subgroup of $\mathbb{Z}_2^{B}.$

\sk 

Let  $\pi \colon S \times Y \to Y$ be an action of a monoid
 $S$ on 
a Stone space $Y$, where at least the translations $\pi_s \colon Y \to Y, y \mapsto sy$ are continuous. It is equivalent to say that the corresponding homomorphism $h \colon S \to C(Y,Y)$ is well defined.   
The functoriality of the Stone and Pontryagin dualities induce
the actions of $S$ on $B$ and $B^*$. 
More precisely, we have the right action 
$$\alpha \colon  B \times S  \rightarrow B, \ \ (\chi s)(x):=\chi(sx)$$
and the left action 
$$\beta: S \times B^{\ast} \rightarrow B^{\ast},\ \ (sf)(\chi):=f(\chi s).$$
Every translation under these actions is a continuous group
endomorphism. 
Moreover, every translation $\a_s \colon B \to B$ ($s \in S$) is even a ring endomorphism. 
Therefore we have the associated monoid anti-homomorphism:
$$i_{\alpha}: S \to \rm{End}_{R}(B)$$
and the monoid homomorphism 
$$i_{\beta}: S \to \rm{End}(B^*),$$
where $\rm{End}_{R}(B)$ 
and  $\rm{End}(B^*)$  are defined as in Remark \ref{r:monoids}.   
Note that $\rm{End}_{R}(B)$ is a topological submonoid of $B^B$ ($B$ is discrete) and $\rm{End}(B^*)$ is a topological submonoid of $C(B^*,B^*)$ in the uniform topology ($B^*$ is compact). 
If $\pi$ is continuous, then one may
show 
that the actions $\a$ and $\beta$ are also jointly continuous (and then the corresponding anti-homomorphism $i_{\a}$ and homomorphism $i_{\beta}$ are continuous). This follows, in particular, from Theorem \ref{p:anti-iso}. 

The pair $(\alpha, \beta)$ is a birepresentation of $S$ on $w \colon B
\times B^* \to \Z_2.$ Meaning that, 
$$\forall s \in S \ \ \  w(\chi s,f)=w(\chi, sf)=f( \chi s).$$

 Define the following \textit{adjoint map} (induced by the Pontryagin duality)
	$$
\Psi \colon \rm{End}(B) \to \rm{End}(B^*), \ \  \mu \mapsto \mu^* \ \ 
\mu^*(f):=f \circ \mu \ \ \forall \mu \in End(B)  \ \ \forall f \in B^* 
$$

which is an anti-isomorphism of monoids by the Pontryagin duality properties. 

\begin{remark} \label{r:delta}  
	It is straightforward to verify that 
	the natural evaluation 
	map
	$$\delta: Y \to B^*, \ y \mapsto \delta_y, \ \ \ \delta_y(\chi)=\chi(y)$$
	is a topological $S$-embedding. In these terms, $\delta(Y) \subset B^*$ is just the subset of 
	all ring homomorphisms $\rm{Hom}_{R}(B,\Z_2)$ in the set of all group homomorphisms 
	$B^*:=\rm{Hom}(B,\Z_2)$ (see \cite[Theorem 32]{GH-Stone}). 
\end{remark}

\sk  
\begin{remark} \label{r:Str} 
	In \cite{MS1} we explore the following known fact
	(see \cite[Theorem 26.9]{HR}) that 
	for every locally compact abelian group $G$ and its Pontryagin dual $G^*,$   
	the canonically defined adjoint map
	between
	$\rm{Aut}(G)$ and $\rm{Aut}(G^*)$  is an  anti-isomorphism of topological  groups (where these automorphism groups equipped with the Birkhoff topology).  
	This is true also for the 
	topological rings of group endomorphisms 
	$\rm{End}(G)$ and $\rm{End}(G^*)$ under the compact-open topology 
	(see \cite[Corollary 25.2]{Str}).
	
\end{remark}

\sk 
An invertible function $f \colon X_1 \to X_2$ between two uniform spaces is said to be a \textit{uniformism} if 
it is a uniform isomorphism meaning that 
both $f$ and $f^{-1}$  are uniform functions.  


\begin{thm} \label{p:anti-iso} 
	Let $Y$ be a Stone space and $B:=C(Y,\mathbb{Z}_2)$ be its (discrete) Boolean ring. Then the canonical monoid anti-isomorphisms  
		\begin{equation} \label{eq:Y} 
		\Phi \colon C(Y,Y) \to \rm{End}_{R}(B), \ \  s \mapsto s^* \ \ 
		s^* (\chi):=\chi \circ s \ \ \forall \chi \in B
	\end{equation} 
		\begin{equation} \label{eq:B*}  
\Delta \colon \rm{End}(B) \to \rm{End}(B^*), \ \  \s \mapsto \s^* \ \ 
\s^*(f):=f \circ \s \ \  \forall f \in B^* 
\end{equation} 
	are uniformisms, where $\rm{End}(B)$ and its submonoid $\rm{End}_{R}(B)$ carry the pointwise uniformity, while $C(Y,Y)$ and $\rm{End}(B^*)$ carry the  uniformity of uniform convergence. 
\end{thm}
\begin{proof}
These maps are well defined and \textbf{bijective} by Pontryagin and Stone duality properties, respectively. 
	 We have the following actions: 
	$$\pi \colon C(Y,Y) \times Y \to Y$$ 
	$$\rm{End}_{R}(B) \times B \to B \ \ (s^* \chi)(x):=\chi(s x) \ \ \ \forall s \in C(Y,Y).$$
Consider the corresponding natural uniformities $\U_1$ and $\U_2$  on $C(Y,Y)$ and  $\rm{End}_{R}(B),$ respectively. 
 The uniformity $\U_1$ 
 is defined in Subsection \ref{d:unif-conv}. 
 Since $Y$ is a Stone space, its uniformity (in terms of uniform coverings \cite{Isb}) is generated by finite clopen partitions. Two element 
 partitions $$\eps_A:=\{\{A,A^c\}: A \in clop(Y)\}$$ define a subbase of the unique compatible uniformity on the Stone space $Y.$  
Taking into account that $B$ is discrete we define $\U_2$ on $\rm{End}_{R}(B) \subset B^B$ as the  pointwise uniformity (see 
Subsection \ref{d:p}).

Fortunately, typical subbase entourages in both cases can be indexed by elements 
of  $B.$ Note that each $\chi \in B$ has the form of the characteristic function $$\chi_A \colon Y \to \Z_2, \chi_A(x)=1 \Leftrightarrow x \in A$$ 
for some clopen subset $A \in clop(Y)$.  
Fix some $\chi:=\chi_A$ with $A \in clop(Y)$ and define 
		$$[\chi]_1:=
		\{(s_1,s_2) \in S \times S: (s_1(x),s_2(x))  \in \eps_A \ \ \forall x \in Y \}= $$
$$=
	\{(s_1,s_2) \in S \times S: \ s_1^{-1}(A) =s_2^{-1}(A) \wedge s_1^{-1}(A^c) =s_2^{-1}(A^c)  \}= $$
	$$ =\{(s_1,s_2) \in S \times S: \ s_1^{-1}(A) =s_2^{-1}(A)\}.$$


\sk 
The family $\{[\chi_A]_1:  A \in clop(Y) \}$ is a subbase of $\U_1$. 
For $\U_2$ a natural subbase is $\{[\chi_A]_2: \  A \in clop(Y) \}$, where  
	$$[\chi]_2:=\{(s_1^*,s_2^*) \in S^* \times S^*: \ s_1^*(\chi) = s_2^*(\chi)\}=$$
	$$  =\{(s_1^*,s_2^*) \in S^* \times S^*: \ \chi \circ s_1 = \chi \circ s_2\}.$$ 

\sk 


 


	Clearly, 	$(s_1,s_2) \in [\chi]_1 \Leftrightarrow (s_1^*,s_2^*) \in [\chi]_2$ because $ s_1^{-1}(A) =s_2^{-1}(A) \Leftrightarrow \chi \circ s_1 = \chi \circ s_2$. This proves that $\Phi \colon C(Y,Y) \to \rm{End}_{R}(B)$ is a uniformism. 
		
	\sk 
	
		Note that 
		$B^{\ast}:=\rm{Hom}(B,\Z_2)\subset \Z_2^B$ carries the pointwise topology. So, 
		the subbase entourage on the compact space $B^*$ naturally defined by the point $\chi \in B$ is  
		$$
		[\chi]^*:=\{(f_1,f_2) \in B^* \times B^*: \ f_1(\chi)=f_2(\chi)\}.
		$$ 
			Now consider the following actions: 
		$$\rm{End}(B) \times B \rightarrow B, \ (s,\chi) \mapsto s(\chi)$$
		$$\rm{End}(B^*) \times B^{\ast} \rightarrow B^{\ast},\ \ (s^*f)(\chi):=f(\chi \circ s).$$ 
	Denote by $\U_3$ the uniformity of uniform convergence on $\rm{End}(B^*) \subset \rm{Unif}(B^*,B^*)$ inherited from $\rm{Unif}(B^*,B^*)$, where the compact space $B^*$ carries the natural uniformity. The corresponding natural uniform subbase again can be parametrized by $\chi \in B$ as follows:   
	$$	[\chi]_3:=\{(s_1^*,s_2^*) \in \rm{End}(B^*) \times \rm{End}(B^*):  
		(s_1^* \psi, s_2^* \psi) \in [\chi]^* \ \   \forall \psi \in B^*\} = 
		$$
		$$ 
		=\{(s_1^*,s_2^*) \in \rm{End}(B^*) \times \rm{End}(B^*):  \ \psi(s_1(\chi))= \psi(s_2(\chi)) \ \  \forall \psi \in B^*\}, 
		$$ 
%
	Since $B^*$ separates the points of $B$ we obtain that
	$$
	s_1(\chi) = s_2(\chi) \Leftrightarrow \psi(s_1(\chi))= \psi (s_2(\chi)) \ \forall \psi \in B^*
	$$
	Therefore 	$s_1(\chi) = s_2(\chi) \Leftrightarrow (s_1^*,s_2^*) \in [\chi]_3$.  This proves that $\Delta \colon \rm{End}(B) \to \rm{End}(B^*)$ is a uniformism. 
\end{proof}

\sk 
In Theorem \ref{p:anti-iso} instead of Boolean rings  one may consider Boolean algebras (and the corresponding endomorphisms) as we mentioned in Remark \ref{r:rings}. 

\begin{cor}  \label{c:isom}  
		Let $Y$ be a Stone space. 
	The homomorphism 
	$$h=\Delta \circ \Phi \colon C(Y,Y) \hookrightarrow \rm{End}(B^*)$$ is an embedding of topological monoids and the pair $(h,\delta)$ is 
	equivariant
	 (meaning that $\delta(f(y))=h(f) (\delta(y))$, where $\delta \colon Y \hookrightarrow B^*$ is an embedding of compact spaces from Remark \ref{r:delta}.  
\end{cor}

\begin{remark} \label{c:antiisomorphic} 
Let $2^{\omega}$ be the Cantor cube. Its Boolean algebra is 
the countably infinite atomless Boolean algebra $B_{\infty}$. 
 Theorem \ref{p:anti-iso} implies that the topological monoids $\rm{End}_{R}(B_{\infty})$ and $C(2^{\omega},2^{\omega})$ are anti-isomorphic.  This fact was proved in \cite{Elliott} using \cite[Corollary 6.12]{Elliott}  (the property of $C(2^{\omega},2^{\omega})$ mentioned in Remark \ref{r:ResOfElliott}.1).
\end{remark}

\sk 
Now we give a \textbf{proof of Theorem} \ref{t:conditSEM}. 

We have to show that the following assertions are equivalent: 

	\ben
	\item $S$ is an l-NA topological monoid.
	
	\item $S$ is a topological submonoid of $C(Y,Y)$ for some Stone space $Y$  (where $w(Y) = w(S)$).
	
			\item  
	The opposite monoid $S^{op}$ can be embedded into 
	the monoid $\rm{End}_{R}(B)$ of endomorphisms of some discrete Boolean ring $B$ (with cardinality $|B| \leq w(S)$). 
	
	\item
	$S^{op}$ is a topological submonoid of $D^{D}$
	for some discrete set $D$ (where $|D|\leq w(S)$).
	
	\item There exists an ultra-metric space $(M,d)$ such that 
	$S^{op}$ is 
	a topological submonoid of the monoid $\Theta(M,d)$  of all 1-Lipschitz maps $M \to M$ equipped with the pointwise topology (where $w(M) \leq w(S)$). 
	
	\item There exists a 
	topologically compatible uniformity $\U$ on $S$ which is generated by a family of right $S$-nonexpansive  ultra-pseudometrics. 
	
	\item
	$S$ is topologically isomorphic to a submonoid of $\rm{Unif}(Y,Y)$ for some 
	NA uniform space 
	$(Y,\mathcal V)$.

	\item 	$S$ can be embedded into the monoid $\rm{End}(K)$ of endomorphisms of some profinite Boolean group $K$  (with $w(K) \leq w(S)$).  
	\item $S$ can be embedded into the monoid $\rm{End}(K)$ of endomorphisms for a compact
	 abelian topological group $K$ (with $w(K) \leq w(S)$).

		%
	\een
\begin{proof} 
	We are going to check that (1) $\Rightarrow$ (2) $\Rightarrow$ (3) $\Rightarrow$ (4) $\Rightarrow$ (5) $\Rightarrow$ (6) $\Rightarrow$ (7) $\Rightarrow$ (1) and (3) $\Rightarrow$ (8) $\Rightarrow$ (9)  
	$\Rightarrow$ (4). 
	\sk 
	
	(1) $\Rightarrow$ (2) 
	By 
	definition, there exists a $0$-dimensional proper $S$-compactification $\nu \colon S \hookrightarrow Y$ of the left action of $S$. 
	The associated continuous monoid homomorphism $h_{\nu} \colon S \to C(Y,Y)$ is a topological embedding because $\nu$ is a topological embedding and the orbit map $h(S) \to \nu(S), \ h(s) \mapsto \nu(s)$ is continuous. 
	As it was mentioned in 
	Remark \ref{r:oppos}.1, one may assume that $w(Y)=w(S)$.	 
	
	(2) $\Rightarrow$ (3)     
$S$ is a topological submonoid of $C(Y,Y)$ for some Stone space $Y$, where $w(Y) = w(S)$. 
Let $B=C(Y,\Z_2)$ be the discrete set of all clopen subsets in the Stone space $Y$. Then $|B|=w(Y)=w(S)$. Now, by Theorem \ref{p:anti-iso}, $C(Y,Y)^{op}$ (hence, also $S^{op}$) 
can be embedded into 
the monoid $\rm{End}_{R}(B)$ with cardinality $|B| \leq w(S)$.
	
	(3) $\Rightarrow$ (4) $S^{op}$ is embedded into $\rm{End}_{R}(B)$ 
	which is a submonoid of $B^B$. So, simply take $D:=B$. 
	
	
	(4) $\Rightarrow$ (5) Consider the two-valued ultra-metric on the discrete space $M:=D$. 
	
	(5) $\Rightarrow$ (6) We have the 
	left action of the opposite semigroup $S^{op} \times M \to M$. 
	For every $z \in M$ consider the  ultra-pseudometric  
	$$
	\rho_z(s,t):=d(s(z),t(z))  \ \ s,t \in S^{op}.
	$$ 
		The collection $\{\rho_z\}_{z \in M}$ generates a compatible zero-dimensional uniformity $\U$ of $S$. Also, 
	$\rho_z(us,ut)=d(us(z),ut(z))\leq d(s(z),t(z))=\rho_z(s,t)$ for every $u, s,t \in S^{op}$. 
	Therefore every $\rho_z$ is left nonexpansive for $S^{op}$. Hence, right nonexpansive for $S$. 
	
	(6) $\Rightarrow$ (7) 
	Let $\{\rho_i\}_{i \in I}$ be a
	family of right $S$-nonexpansive ultra-pseudometrics on $S$ which generates a  (necessarily, zero-dimensional) 
compatible	uniformity $\U.$
	Consider the left action $S \times S \to S$. Then 
	$\rho_i(su,tu) \leq \rho_i(s,t)$ for every $s,t,u \in S$. This implies that $\U$ is bounded (in the sense of Definition \ref{d:equiun})  with respect to the given left action. 
By Proposition \ref{l:noteasy}, this action is 
$\mathcal V$-equiuniform,
 where 
$\mathcal V=\U_S$ is an NA uniformity.
	Therefore, by Definition \ref{d:equiun}.3, 
	the associated monoid homomorphism $h \colon S \hookrightarrow \rm{Unif}(Y,Y)$ is well defined and continuous, where $Y=S$ is equipped with the uniformity $\mathcal V.$ Since $S$ is a monoid, $h$ is injective. The continuity of 
	$f\colon h(S) \to S, \ h(s) \mapsto s \cdot e =s$  follows from the containment $f(\widetilde{\eps}(h(s))\cap h(S))\subset \eps(s)$. This ensures that $h$ is an embedding of topological monoids.

		(7) $\Rightarrow$ (1) Apply Theorem \ref{t:UNIF}. 
	
%
	 
		(3) $\Rightarrow$ (8) Take $K:=B^*$ and apply Theorem \ref{p:anti-iso}. Note that $w(B^*)=w(Y)$. 
		
		(8) $\Rightarrow$ (9)	Every profinite Boolean group is compact abelian. 
		
 
(9) $\Rightarrow$ (4) Use the anti-isomorphism $\Psi \colon \rm{End}(K) \to \rm{End}(K^*)$ 
of topological rings (in particular, of topological monoids)  
from Remark \ref{r:Str} 
 and the fact that $\rm{End}(K^*)$ is a topological submonoid of $D^{D},$ where $D=K^\ast$ is discrete.  
	%
	%
	%
\end{proof}

\sk 

Recall that a topological group is NA if and only if it has a local base at the identity consisting of (open) subgroups. As a Corollary of Theorem \ref{t:conditSEM}, we obtain the following implication.
\begin{cor}\label{cor:opsub}
Let $S$ be a topological monoid. If   $S$ is either l-NA or r-NA, then it has a local base at the identity consisting of open submonoids.
\end{cor}
\begin{proof}
Assume first that $S$ is an l-NA topological monoid. By the equivalence (1) $\Leftrightarrow$ (6), there exists a
family $\{\rho_i\}_{i \in I}$ of right $S$-nonexpansive ultra-pseudometrics    on $S$ which generates its topology. For every $i\in I$ and $r>0$ the open ball $B_{\rho_i}(e,r):=\{x: \ \rho_i(x,e)<r\}$ is a submonoid of $S.$ Indeed, since $\rho_i$ is right $S$-nonexpansive ultra-pseudometric it holds that $$\rho_i(xy,e)\leq \max \{\rho_i(xy,y),\rho_i(y,e) \}\leq \max \{\rho_i(x,e),\rho_i(y,e)\}.$$
This implies that $S$ has a local base at the identity consisting of open submonoids. The case of an  r-NA topological monoid can be proved similarly taking into account that $B_{\rho_i}(e,r)$ is a submonoid of $S$ also when the ultra-pseudometric  $\rho_i$ is left $S$-nonexpansive.
\end{proof}

\begin{remark} \label{r:notNA} 	
There exists a locally compact metrizable separable zero-dimensional  
topological monoid $S$ that has a local base at the identity consisting of open submonoids which is neither  l-NA nor r-NA. Indeed, one may use
the monoid $S$ from 
 Example \ref{ex:contrast} to construct topological monoids $S_1:=S$ and $S_2:=S^{op}$ such that $S_1$ is  l-NA but not r-NA while $S_2$ is  r-NA and not l-NA .  By Corollary \ref{cor:opsub}, 
both $S_1$ and $S_2$ have a local base at the identity  consisting of open submonoids. Then, the topological monoid $S_1\times S_2$ has the latter property while it is neither l-NA nor r-NA (even, not l-compactifiable and nor r-compactfifiable).

In contrast, note that every locally compact zero-dimensional topological \textit{group} must be NA as it follows by a classical result of van Dantzig (see, for example, \cite[Theorem 7.7]{HR}). 
\end{remark}

\sk 
Recall again that the Polish symmetric group $S(\N) \subset \N^{\N}$ is  universal for second countable NA topological groups. The following result is a natural analog for NA monoids. 

\begin{thm} \label{c:universal} The Polish monoid 
	$\N^{\N}$ is a universal separable metrizable 
	r-NA monoid. 
	More generally, $\kappa^{\kappa}$ is a universal r-NA monoid of weight $\kappa$ for every infinite cardinal $\kappa$. 
\end{thm}
\begin{proof}
	If $S$ is second countable then  one may assume that $B$ in Theorem \ref{t:conditSEM}  is countable. So, $S$ is embedded into $\rm{End}_{R}(B)$ which, in turn, is embedded into $B^B \simeq \N^{\N}$. 
\end{proof}

\sk 
By results of \cite{MeSc01}, $\Homeo(2^{\omega})$ is a universal Polish NA group. Moreover, the action of $G:=\Homeo(2^{\omega})$ on the Cantor cube $2^{\omega}$ is
universal in the class of all actions $G \times Y \to Y$, where $Y$ is a metrizable Stone space and $G$ is a topological subgroup of $\Homeo (Y)$.  
More precisely, there exists an  equivariant pair $(h,\a)$, where $h \colon G \rightarrow \Homeo(2^{\omega})$ is an embedding of topological groups and $\a \colon Y \hookrightarrow 2^{\omega}$ is a topological embedding.

In fact, the following \textit{monoid version} holds.

\begin{thm} \label{c:CANTOR} 
	The Polish monoid $C(2^{\omega},2^{\omega})$ is universal for separable metrizable 
	l-NA monoids. Moreover, the action of $C(2^{\omega},2^{\omega})$ on $2^{\omega}$ is 
	universal in the class of all actions $S \times Y \to Y$, where $Y$ is a metrizable Stone space and $S$ is a topological submonoid of $C(Y,Y)$.  
\end{thm}
\begin{proof}
	For every countable infinite Boolean ring $B$ the corresponding Pontryagin dual $B^*$ of the group $B$ will be a compact zero-dimensional, compact metric space. Since $B^*$ is a (compact) topological group and infinite it has no isolated points. 
	Recall that 
	a classical  (since 1910) theorem of Brouwer characterizes the Cantor space $2^{\omega}$ as the unique zero-dimensional, compact metric space without isolated points. 
	Now, choose $B:=C(Y,\Z_2)$ as the Boolean ring of $Y$. 
	So, using again the Stone duality and Corollary \ref{c:isom} we complete the proof. 
\end{proof}


\begin{prop} \label{p:emb}  \ 
	 \begin{enumerate}
	 	\item 	  $C(2^{\omega},2^{\omega})$ is embedded into $(\N^{\N})^{op}$ and $(\N^{\N})^{op}$ is embedded into $C(2^{\omega},2^{\omega})$.  
	
	 	\item $(\N^{\N})^{op}$ and $C(2^{\omega},2^{\omega})$ are not isomorphic as topological monoids.
	 \end{enumerate} 
\end{prop}
\begin{proof}
	(1) It is an immediate corollary of Theorems \ref{c:universal} and \ref{c:CANTOR}.
	
	(2) 
	It is well known that the symmetric group $S_{\N}$ is universal for NA second countable groups. The same is true for the group of homeomorphisms $\Homeo (2^{\omega})$, \cite{MeSc01}. Therefore, $S_{\N}$ is embedded into $\Homeo (2^{\omega})$ and also 
	$\Homeo (2^{\omega})$ is embedded into $S_{\N}$. However, the universal minimal dynamical systems of these groups have completely different nature according to their concrete descriptions due to Glasner and Weiss \cite{GW02, GW03}.  
	It follows that these groups are not topologically isomorphic. One more conclusion of this observation is that $(\N^{\N})^{op}$ and $C(2^{\omega},2^{\omega})$ are not isomorphic as topological monoids. The reason is that their subgroups of invertible elements are just $S_{\N}$ and $\Homeo (2^{\omega})$. On the other hand, any isomorphism of monoids induces an isomorphism of their groups of invertible elements. 
\end{proof}

\begin{remark} \label{r:CompSem} Here we give some examples of lr-NA monoids  (Definition \ref{d:NA-action}.3).   
	\begin{enumerate}
		\item 
		Any compact 
		zero-dimensional topological monoid $S$ is lr-NA. 
		In particular, $S$ is 	
		embedded into the monoid $\kappa^{\kappa}$. If, in addition, $S$ is metrizable, then it is embedded into the Polish monoid $\N^{\N}$; this implies \cite[Theorem 1.5]{Bard}. Such $S$ is also embedded into $C(2^{\omega},2^{\omega})$ by Theorem \ref{c:CANTOR}.  
		\item 
		Any product of discrete monoids is lr-NA. 
		It is enough to prove the case of a discrete monoid $S$ 
		(the class lr-NA is productive). Since such $S$ is opposite of the discrete monoid $S^{op}$, it is enough to show that $S$ is r-NA. In order to see this observe that the Cayley homomorphism $h \colon S \to (S^S, \tau_p)$ (by the left translations) is an embedding of topological monoids. This implies the embedding of a countable product of countable discrete monoids into $\N^{\N}$ which was proved in \cite[Lemma 2.2]{Bard}. 
	\end{enumerate} 
\end{remark}

Note that by \cite[Proposition 3.6]{Bard}, there exists a locally compact Polish countable topological monoid which cannot be embedded into the topological monoid $\N^{\N}$. This answers Question 5.6 from \cite{Elliott}.

\begin{example} \label{ex:contrast} 
	There exists a locally compact 
	metrizable separable zero-dimensional 
	topological monoid $S$ which is r-NA but not l-NA (even not l-compactifiable). 
	Indeed, 
	consider the 2-point multiplicative monoid $\{0, 1\}$ and endow the Cantor cube
	$C :=\{0, 1\}^{\N_0}$ with the topological monoid structure of
	pointwise multiplication. Let $\N_0 := \N \cup \{0\}$. Consider the following 
	continuous left action 
	$$\pi\colon  C \times \N_0 \to \N_0, \ \  \pi(c,n)=c_nn,$$
	where $c=(c_k)_{k \in \N_0} \in C$. 
	
	 Denote by $S:=C \sqcup_{\pi} \N_0$ a new monoid 
	defined as follows. As a topological space it is a {\it disjoint sum} $C \cup
	\N_0 $. The multiplication is defined by setting:
	
	$a \circ b:=\pi(a,b):=a_n b$  if $a \in C , \ b \in \N_0 $
	
	$a \circ b:=ab$  if $a \in C , \ b \in C $
	
	and
	
	$a \circ b:=a$ if $a \in \N_0 \ \forall b \in S$. 
	
		Clearly, ${\mathbf 1}:=(1,1,
	\cdots)$ is the identity of $S$. Observe that 
	for every neighborhood $U$ of ${\mathbf 1}$ we have $0 \in U \N$. 
	
Let $\rho$ be the standard ultra-metric on the Cantor cube $C$ defined for every $s,t \in C$ as 
$$\rho(s,t):=\frac{1}{\min\{n \in \N: \ s_n \neq t_n\}}.$$ 
 It is non-expansive under left (right) translations. Extend it to  
	a compatible ultra-metric $d$ on $S=C \sqcup \N_0$ as follows: 
	$d(s,t)=\rho(s,t)$ for every $s,t \in C$ and $d(s,t)=1$ for every other cases with distinct $s,t$. 
	Then $d$ is a compatible ultra-metric on $S$ which is left nonexpansive. 
	By Fact \ref{f:BodSchn}, 
	$S$ is r-NA being embedded into the topological monoid $\N^\N$.

	Assuming that $S$ is left-compactifiable, 
	there exists a proper $S$-compactification $\nu \colon  S \hookrightarrow Y$ of the left action $S \times S \to S$. For simplicity we identify $S$ and $\nu(S)$. Consider $\N=\nu(\N)  \subset S$, a closed subset of $S$ and $0 \in S$ with $0 \notin \N$. Then $0 \notin cl_Y (\N)$. Clearly, $K:=cl_Y (\N)$ is a compact subset of $Y$.  
	Using the standard compactness argument and 
	the continuity of the action, there exists a neighborhood $U$ of ${\mathbf 1} \in S$ such that $0 \notin UK$. Then $0 \notin U\N$, a contradiction. 	
\end{example}

The following lemma is a non-archimedean adaptation of some classical facts about uniform spaces and pseudometrics 
going back to A. Weil (see, for example, \cite[Metrization Lemma 0.29]{RD}).  

\begin{lem} \label{l:ultra} \ 
	\begin{enumerate} 
		 \item Let $X$ be a set and $F:=\{\s_n : n \in \N\},$ with $\s_{n+1} \subset \s_n,$ be a countable monotone family 
		 of equivalence relations.
		 Then there exists an ultra-pseudometric $d_F$ on $X$ such that 
		$$
		\s_{n+1} \subset \{(x,y) \in X^2: d(x,y) < 2^{-n}\} \subset \s_n.
		$$
		\item Let $\U$ be a pre-uniformity on a set $X$. Then $\U$ is NA if and only if  
		 there exists a family $\g:=\{d_i: i \in I\}$ of ultra-pseudometrics on $X$ (with $d_i \leq 1$ for every $i \in I$) which generates $\U$. Moreover, if $\g$ is countable, then there exists an ultra-pseudometric $d$ on $X$ which generates $\U$. 
		
	\end{enumerate} 
\end{lem}
\begin{proof}
	(1) Let $\s_0:=X \times X$ be the tautological equivalence relation. 
	The desired ultra-metric is  
	$$
	d_F(x,y):=\inf \big\{\max_{0\leq i\leq k}c(x_i,x_{i+1}): \ k \in \N, x_i \in S, x_0=x, x_{k+1}=y\big\}, 
	$$
	where $c(x,y):=2^{-n}$ if there is $n \in \N \cup \{0\}$ such that $(x,y) \in \s_n \setminus \s_{n+1}$ 
and $c(x,y):=0$ otherwise.
	
	(2) Easily follows from (1) (taking into account that for every ultra-pseudometric $\rho$ the formula defines a new ultra-pseudometric such that $\rho$ and $\rho^*:=\min \{\rho,1\}$ generate the same pre-uniformity). 
	
	If $\g:=\{d_n: n \in \N\}$ is countable with $d_n \leq 1$ then $d:=\sup\{d_n: n \in \N\}$ is the desired ultra-pseudometric.  	
\end{proof}

\sk 

Let $\s$ be an equivalence relation on a monoid $S$. We say that $\s$ is a \textit{left congruence} if left translations preserve $\s$. 

\begin{prop} \label{p:congr} 
	The following assertions are equivalent: 
	\ben
	\item $S$ is an r-NA topological monoid. 
	
	
		\item There exists a zero-dimensional uniformity $\U$ on $S$ which is generated by a family $\{\rho_i: i \in I\}$ of left 
	$S$-nonexpansive ultra-pseudometrics.  
	
	\item There exists a zero-dimensional topologically compatible uniformity $\U$ on $S$ 
 with a 
	 basis  $\gamma:=\{\s_i: i \in I\}$ which consists of equivalence relations $\s_i$, where each $\s_i$ is a left congruence.  

	\een
\end{prop}
\begin{proof} 
	(1) $\Leftrightarrow$ (2) By the equivalence (1) $\Leftrightarrow$ (6) in the dual version of Theorem \ref{t:conditSEM}.  
	there exists a zero-dimensional topologically compatible uniformity $\U$ on $S$ which is generated by a family $\{\rho_i: i \in I\}$ of left $S$-nonexpansive ultra-pseudometrics on $S$.

	(2) $\Rightarrow$ (3)  
	$\s_{i r}:=\{(x,y): \rho_i(x,y)<r\}$ is an equivalence relation for every $i \in I$ and $r>0.$  
 Since $\rho_i$ is left nonexpansive, then $\s_{i r}$ is a left congruence.  
 
 \sk 
 (3) $\Rightarrow$ (2) For every $\s_i$ we have the associated ultra-pseudometric defined by $\rho_i(x,y)=1$ for every $\s_i$-equivalent elements $x, y$. 
 Then $\{\rho_i: i \in I\}$ is a compatible family of left 
 $S$-nonexpansive ultra-pseudometrics on $S$. 
\end{proof}

It is well known (Lemin \cite{Lem84}) that a metrizable topological group $G$ is NA if and only if  $G$ admits a left invariant ultra-metric.  

\begin{defin}
	Let us say that 
	a topological monoid $S$ is \textit{l-ultrametrizable}  if there exists a topologically compatible left $S$-nonexpansive  ultra-metric $d$ on $S$. Similarly can be defined \textit{r-ultrametrizable} monoid. 
\end{defin}

As in the proof of Lemma \ref{l:ultra}, one may show that $S$ is l-ultrametrizable if and only if there exists a \textit{countable} 
	compatible
	family $\gamma:=\{\s_n: n \in \N\}$ of left $S$-nonexpansive ultra-pseudometrics.  

In particular, by Proposition \ref{p:congr} (for countable $\g$)  and Lemma  \ref{l:ultra}, we obtain that every second countable l-NA monoid $S$ is l-ultrametrizable. 
Using also Theorem \ref{c:universal},   the following known result, covered by \cite[Theorem 2.1]{BodSchn},  is obtained.
\begin{f}  \label{f:BodSchn} 
	(M. Bodirsky and F.M. Schneider \cite{BodSchn}) 
	 Let $S$ be a second countable topological monoid. 
	The following conditions are equivalent:
	\begin{enumerate} 
			\item $S$ is l-ultrametrizable; 
			
		\item $S$ is embedded, as a topological monoid, into $\N^{\N}$;  
	
		\item There exists a zero-dimensional topologically compatible uniformity $\U$ on $S$ with a countable uniform basis  $\gamma:=\{\s_n: n \in \N\}$ which consist of equivalence relations $\s_n$, where each $\s_n$ is a left $S$-congruence with $|\s_n| \leq \aleph_0$.
	\end{enumerate}
\end{f}

%
%

	Let $\F$ be an NA valued field and $(V,||\cdot||)$ be an ultra-normed space over $\F$. Then the topological monoid $\Theta_{lin}(V)$ (being a submonoid of $\Theta(V,d_{||\cdot||})$) is r-NA by Theorem \ref{t:conditSEM}.   In fact, every r-NA monoid is a topological submonoid of  $\Theta_{lin}(V)$ for such $V$, as the next proposition shows. 
	For the definition and examples of NA valued fields see, for example,  \cite{MeShNAtransp}

\begin{prop} \label{p:NAvalued} 
Let $S$ be an r-NA monoid. Then for every  NA valued field $(\F,|\cdot|)$ there exists an ultra-normed $\F$-vector space $(V,||\cdot||)$ such that
$S$ is a topological submonoid of $\Theta_{lin}(V).$
\end{prop}
\begin{proof}
In view of Theorem \ref{t:conditSEM},  it suffices to prove the assertion for $S=\Theta(M,d),$ where $(M,d)$ is an ultra-metric space. Moreover, we can assume that 
 $d\leq 1,$ as $\Theta(M,d)$ is a topological submonoid of $\Theta(M,\rho),$ where $\rho=\min\{d,1\}.$ Let 		$V:=L_\F(M)$ be the  free $\F$-vector space on the set $M$ and $\overline{M}:= M\cup \{\mathbf{0}\},$ where $\mathbf0\notin M$ is the zero element of  $V.$ By \cite[Lemma 4.1]{MeShNAtransp}, one can extend $d$ to an ultra-metric
 on $\overline{M}$ by letting $d(x,\mathbf0)=1$ for every $x\in M.$ Now, let $||\cdot||$ be the  maximal ultra-norm on $V$ extending $d.$   Recall that this is just the Kantorovich ultra-norm  associated with $d$ (see \cite[Definition 4.2]{MeShNAtransp}). We will show that for every $f\in \Theta(M,d)$ it holds that $\overline{f}\in\Theta_{lin}(V)$, where $\bar f\colon V\to V$ denotes  the  linear extension of $f$.  Since $f\in \Theta(M,d)$ and $d(y,\mathbf0)=1$ for every $y\in M$ it holds that $\bar{f}\in \Theta(\overline{M},d).$  Let $v\in V$ and let us show that $||\bar{f}(v)||\leq ||v||.$ 
 By Theorem 4.2 
 (NA Arens-Eells embedding) and Theorem 5.2 
 (Min-attaining Theorem) of \cite{MeShNAtransp}, 
 $$
 ||v||=\max_{1\leq i\leq n}|\lambda_i|d(x_i,y_i),
 $$ 
 for some representation $$v=\sum_{i=1}^n\lambda_i(x_i-y_i), \ x_i,y_i\in \overline{M}, \ \lambda_i\in \F.$$
 Using the linearity of $\bar{f}$ we have $$\bar{f}(v)=\sum_{i=1}^n\lambda_i(\bar{f}(x_i)-\bar{f}(y_i)).$$
At this point, recall that ${f}\in \Theta(M,d)$ and $d(f(x),f(\mathbf 0))=d(x,\mathbf 0)=1$ for every $x\in M.$ So,  using \cite[Theorem 4.3]{MeShNAtransp} again we deduce that $$||\bar{f}(v)||\leq \max_{1\leq i\leq n}|\lambda_i|d(\bar{f}(x_i),\bar{f}(y_i))\leq \max_{1\leq i\leq n}|\lambda_i|d(x_i,y_i)=||v||.$$ 
 
As the monoids  $\Theta(M,d)$ and $\Theta_{lin}(V)$ are equipped with the pointwise topology and by the interrelations between the ultra-metric $d$ and the ultra-norm $||\cdot||,$ we conclude that the assignment $f\mapsto \bar{f}$ is a topological  embedding of  $\Theta(M,d)$ into $\Theta_{lin}(V).$ 
\end{proof}

\sk 
\section{Appendix: a factorization theorem for monoid actions}

There are several useful factorization and approximation theorems for topological group actions. See \cite{Me-sing89}, \cite{Me-diss85} and \cite{Me-b}. Some of them 
can easily be adopted (sometimes under more restrictive assumptions) for topological monoid actions. 
Theorem \ref{t:G-SklyarenkoSEMmetr} below is one of such results. For the sake of completeness we include here its proof. The proof uses the definition of uniform spaces 
in terms of uniform coverings (see, for example, \cite{Isb} and Definition \ref{d:cov-unif} below). 
We first recall some related definitions. 

\sk 
Let $A$ be a subset of $X$ and $P$ 
be a family of subsets 
of $X$. We write $A \succ P$ if $A$ is a subset of at least one $B \in P$. 
Let $P$ and $Q$ be two coverings of a set $X$. We say that $Q$ is a \textit{refinement} of $P$ and write 
$Q \succ P$ if $A \succ P$ for every $A \in Q$. 
Define also 
$$
P \wedge Q:=\{A \cap B: A \in P, B \in Q\}. 
$$

Let $P$ be a covering of a set $X$ and let $A \subset X$. The \textit{star of $A$ with respect to $P$} is the set  
$$
st(A,P)=\cup\{U \in P: U \cap A \neq \emptyset\}. 
$$
For the singleton $A:=\{a\}$ we simply write $st(a,P)$. 
So, $st(a,P)=\cup\{U\in P: a\in U\}$. 
The collection $$P^*:=\{st(A,P): A \in P\}$$ is a covering and is called the \textit{star of $P$}. Always, $P \succ P^*$. 
If $P^* \succ Q,$ then we say that $P$ is a star-refinement of $Q$. Sometimes we write $P \succ_* Q$ instead of $P^* \succ Q$.  

     For a covering $P$ of $X$ we define the order $ord_x(P)$ and $ord(P)$  by 
$$ord_x(P):=|\{A \in P: x \in A\}| \ \text{and} \  
ord(P):=\sup \{ord_x(P): x \in X\}.$$
If $\cU$ contains a base consisting of covers $P$ with $ord(P) \leq n+1$, where $n$ is a given nonnegative integer, then we say that the (uniform) dimension $\dim(\cU) \leq n$.  
We write $\dim(\cU) = \infty$ if $\dim (\cU) \geq n$ for every $n \in \N$. 
For compact spaces this gives just the usual topological covering dimension $\dim$. 
Note that the notation of the uniform dimension in \cite[Ch. V]{Isb} is $\Delta d$ and if it is finite then $\dim(\cU) \leq n < \infty$ if a and only if every finite uniform covering has a finite uniform refinement with order $\leq n+1$. The completion and the Samuel compactification both preserve the dimension. 

\sk 
If $S \times X \to X$ is an action, then for every subset $A \subset S$ and a family $P$ of subsets in $X$ define 
$
AP:=\{AU: U \in P\}. 
$

\begin{defin} \label{d:cov-unif}[coverings approach]  \cite{Isb} 
	Let $\cU$ be a family of coverings  on a set $X$. 
	Then $\cU$ is said to be a \textit{(covering) pre-uniformity} on $X$ if: 
	\begin{itemize}
		\item [(C1)] $P,Q \in \cU$ implies that $P \wedge Q \in \cU$; 
		\item [(C2)] $P \in \cU$ and $P \succ Q$ imply that $Q \in \cU$; 
		\item [(C3)] 
		for every  $Q \in \cU$ there exists $P \in \cU$ such that $P^* \succ Q$.   
	\end{itemize} 
	$\cU$ is a \textit{uniformity} (Hausdorff pre-uniformity) if for every distinct points $x,y \in X$ there exists $P \in \cU$ such that $st(x,P)$ and $st(y,P)$ are disjoint.  
\end{defin}
%
%
%

\begin{remark} \label{r:BoundCover}  \ 
	\begin{enumerate}
		\item As to the link between 
		theese two approaches, note that 
		every uniform covering $P \in \cU$ induces the corresponding entourage 
		$\tilde{P}:=\cup \{A \times A: A \in P\}.$ 
		Every entourage $\eps \in \U$ induces the corresponding \textit{$\eps$-uniform cover} 
		$\{\eps(x): x \in X\},$ 
		where 
		$\eps(x):=\{y \in X: (x,y) \in \eps\}.$ 
		
		\item  	
		In terms of covering uniformity $\cU$ we have the following condition for $\cU$-equiuniform action (compare Definition \ref{d:equiun}.3): 
		
		\nt	for every covering $P \in \cU$
		there exist a neighborhood $V \in N_{s_0}$ and a covering $Q \in \cU$ such that $V Q \succ P$ for each $x\in X$ and $s_1, s_2 \in V$.   
	\end{enumerate} 
\end{remark}

\begin{thm} \label{t:G-SklyarenkoSEMmetr} Let $\nu \colon X \hookrightarrow Y$  be a proper $S$-compactification of $X$. Then there exists a proper $S$-compactification $\sigma \colon X \hookrightarrow K$ which is majored by $\nu$
		\begin{equation*}
		\xymatrix { X \ar[dr]_{\s} \ar[r]^{\nu} & Y  \ar[d]^{q} \\
			& K }
	\end{equation*} 
	 such that $w(X) \leq w(K) \leq w(X) \cdot w(S)$ and $\dim K \leq \dim Y$. In particular, if $\dim Y=0$ then also $\dim K=0$. 
	 If $w(S) \leq w(X)$ (e.g., if $S$ is second countable), then $w(X)=w(K)$. 
\end{thm}
\begin{proof} 
	Denote by $\cU$ the 
	unique compatible 
	covering uniformity
	  of the compact space $Y$. 
	 There exists 
	a subfamily $\g \subset \cU$ such that $\g$ separates points and closed subsets of $\nu(X)$ (which is homeomorphic to $X$) and $|\g|=w(X)$.  
	
	We claim that there  exists a (not necessarily, Hausdorff) coarser pre-uniformity $\widetilde{\cU}$ on the set $Y$ such that the following four conditions are satisfied:
	\begin{itemize}
		
		\item [(a)] $\widetilde{\cU}$ is bounded and saturated (Definition \ref{d:equiun}.4) with respect to the given action; 
		\item [(b)] $\g \subset \widetilde{\cU} \subset \cU$; 
		\item [(c)] $\dim (\widetilde{\cU}) \leq \dim(\cU)$;  
		\item [(d)] $w(\widetilde{\cU}) \leq |\g| \cdot w(S)$. 
	\end{itemize} 
	
	\sk 
	Since $Y$ is compact and the action of $S$ on $Y$ is continuous, $\cU$ is equiuniform in the sense of Definition \ref{d:equiun}. 
	Therefore, 
	by Remark \ref{r:BoundCover}, 
	 for every pair $P_{\a},P_{\beta} \in \g$ and every $s \in S$ we can choose a covering 
	$P^{s}_{\a \beta} \in \cU$ and a neighborhood $V^s_{\a \beta} \in N_{s}$ of $s$ in $S$ such that: 
	
		
		(1)  $V^s_{\a \beta}P^{s}_{\a \beta} \succ_* P_{\a} \wedge P_{\beta}$. 
		
		
		Moreover, we can assume, in addition, that 
		
		(2) $ord(P^s_{\a \beta}) \leq \dim(\cU) +1$.    
\sk 
	
	Clearly, $S=\cup\{V^s_{\a \beta}: \ s \in S\}$. 
	One may choose a subset $S_1 \subseteq S$ such that $|S_1| \leq w(S)$ and $S=\cup \{V^s_{\a \beta}: \ s \in S_1\}$.  
	Consider 
	$$\mathfrak{B}_1=\{p^s_{\a \beta}: \ s \in S_1\}.$$ 
	Then  the following two additional conditions hold. 
	
	(3) $\mathfrak{B}_1 \subset \cU$ and $\mathfrak{B}_1 \leq |\g|  \cdot w(S)$;  
	
	(4) for every pair $P_{\a},P_{\beta} \in \g$ and for every $s \in S$ there exist $O_{\a \beta} \in N_s$ and $P^s_{\a \beta} \in \mathfrak{B}_1$  
	such that 
	$O_{\a \beta} P^s_{\a \beta} \succ_* P_{\a} \wedge P_{\beta}$. 
	
	\sk 
	We continue by induction. Let us assume that the families $\mathfrak{B}_1, \mathfrak{B}_2, \cdots, \mathfrak{B}_n$ are  already defined. Applying the similar procedure to the family $\cup_{i=1}^n \mathfrak{B}_i$ we get $\mathfrak{B}_{n+1}$.  	
	Note that, in particular, we have 
	
	(5) for every pair $P_{\a},P_{\beta} \in \mathfrak{B}_n$ and for every $s \in S$ there exist $O_{\a \beta} \in N_s$ and $P^s_{\a \beta} \in \mathfrak{B}_{n+1}$  
	such that 
$O_{\a \beta} P^s_{\a \beta} \succ_* P_{\a} \wedge P_{\beta}$. 
	
	\sk 
	The resulting family $\mathfrak{B}=\cup_{i \in \N} \mathfrak{B}_n$ is a base for the following pre-uniformity 
$$\widetilde{\cU}:=\{P : \exists Q \in \mathfrak{B} \ \ \ Q \succ P\}$$ 
	on $Y$. This is the desired  pre-uniformity. Indeed, 
	from the construction 
	it is immediate to see that $\g \subset \widetilde{\cU} \subset \cU$. This proves (b). The conditions (c) and (d)  	
	%
	%
	are also clear.  
	For (a) use condition (5). 
	
	\sk 
	Now, let $K:=(Y^*,\widetilde{\cU}^*)$ be the\textit{ associated (quotient) Hausdorff uniform space} 
	(see Kulpa \cite{Kulpa73} or \cite{Me-sing89, Me-b}) of the pre-uniform space $(Y,\widetilde{\cU})$ and $q \colon  Y \to K$ is the canonical onto map.  
	Then 
	$\widetilde{\cU}^*$ is bounded and saturated, too.  
	Hence, the action $S \times K \to K$ is continuous (Lemma \ref{e:easy}.1). Also, $\dim K=\dim (\widetilde{\cU}) \leq  \dim(\cU)=\dim Y$ and $w(K) \leq w(X) \cdot w(S)$.  
	
	Then 
	$\sigma=q \circ \nu \colon X \to K$ 
	defines the desired proper $S$-compactification of $X$.  
\end{proof}



\bibliographystyle{amsplain}

\begin{thebibliography}{10}

\bibitem{BH97} R.N. Ball and J.N. Hagler,
{\it Real-valued functions on flows}, Preprint, 1997. 

\bibitem{Bard} 
S. Bardyla, L. Elliott, J. D. Mitchell and  Y. Peresse, 
\textit{Topological embeddings into transformation monoids}, 
Forum Mathematicum, 2024 (published online), 
arXiv:2302.08988v2. 


\bibitem{BodSchn} 
M. Bodirsky and F.M. Schneider, 
\textit{A topological characterisation of endomorphism monoids of countable structures}, 
Algebra Universalis, 77(3):251--269, 2017. 

\bibitem{Bourb1}
N. Bourbaki, General Topology, Parts 1, 2, 
Hermann, Paris, 1966. 


\bibitem{Br}
R.B. Brook, {\it A construction of the greatest ambit}, Math.
Systems Theory 6 (1970) 243--248.

\bibitem{Dieudonne} 
J. Dieudonne, Sur la completion des groupes topologique, C.R. Acad. Sc. Paris 1944.  

\bibitem{DPS} D. Dikranjan, Iv. Prodanov and L. Stoyanov, \emph{Topological
	Groups: Characters, Dualities and Minimal Group Topologies}, Pure
and Appl. Math. 130, Marcel Dekker, New York-Basel, 1989.


\bibitem{Elliott} 
L. Elliott, J. Jonusas, Z. Mesyan, J.D. Mitchell, M. Morayne and Y. Peresse, \textit{Automatic continuity, unique Polish
topologies, and Zariski topologies on monoids and clones}, 
Trans. Amer. Math. Soc. 377:3 (2023). 

\bibitem{GW02} 
E. Glasner and B. Weiss, \textit{Minimal actions of the group $S(\Z)$ of permutations of the integers,} Geom. and Funct. Annal. 12 (2002), 964--988. 

 \bibitem{GW03} 
 E. Glasner and B. Weiss, 
 \textit{The universal minimal system for the group of homeomorphisms of the Cantor set,} Fundam. Math. 176 (2003), 277--289. 

\bibitem{GH-Stone} 
S. Givant and P. Halmos, \textit{Introduction to Boolean Algebras,}  
Springer, 2009. 

\bibitem{GM-prox10} L. Google and M. Megrelishvili,
\emph{Semigroup actions: proximities, compactifications and
	normality,} Topology Proc. 35 (2010) 37--71.

\bibitem{HR} E. Hewitt and K.A. Ross, \emph{Abstract Harmonic Analysis
	I}, Springer, Berlin, 1963.

\bibitem{Isb} J. Isbell, \emph{Uniform spaces}, Providence, 1964.

\bibitem{KPT} 
A.S. Kechris, V.G. Pestov and S. Todorcevic, \emph{Fra\"{i}ss\'{e} limits, Ramsey theory, and topological dynamics of automorphism groups}, 
GAFA, \textbf{15} (2005), 106--189.

\bibitem{Kulpa73} 
W. Kulpa, Maps and inverse systems of metric spaces spaces, Fund. Math., 80 (1973), 141--151. 

\bibitem{Lem84}
A.Yu. Lemin, \emph{Isosceles metric spaces and groups,} in:
Cardinal invariants and mappings of topological spaces,, Izhevsk, 1984, 26--31.

\bibitem{Me-EqComp84}
M. Megrelishvili,
{\it Equivariant completions and compact extensions},
Bull. Ac. Sc. Georgian SSR 115:1 (1984) 21--24.

\bibitem{Me-diss85}
M. Megrelishvili,
{\it Uniformity and Topological Transformation Groups} (Russian),
Ph.D. Thesis, Tbilisi State University, 1985, 125 pp.





\bibitem{Me-Ex88}
M. Megrelishvili,
{\it A Tychonoff G-space not admitting a compact Hausdorff
	G-extension or a G-linearization},
Russian Math. Surveys 43:2 (1988) 177--178.

\bibitem{Me-sing89}
M. Megrelishvili,
{\it Compactification and factorization in the category of
	G-spaces}, in: Categorical Topology, ed. J. Adamek and S. MacLane,
World Scientific, Singapore, 1989, 220--237.




\bibitem{Me-cs07}
M. Megrelishvili,
\emph{Compactifications of Semigroups and Semigroup Actions},
Topology Proc., \textbf{31}, no 2 (2007) pp. 611--650.



\bibitem{Me-b}
M. Megrelishvili, 
\textit{Topological Group Actions and Banach Representations}, unpublished book, 2023. Available on author's homepage. 


\bibitem{MeSc01}
M. Megrelishvili and T. Scarr,
{\it The equivariant
	universality and couniversality of the Cantor cube},
Fundamenta Math. {\bf 167} (2001), 269--275.

%

\bibitem{MS1} M. Megrelishvili and M. Shlossberg, \emph{Notes on non-archimedean topological groups}, Top. Appl.,
\textbf{159} (2012), 2497--2505.  Arxiv.org/abs/1010.5987.


\bibitem{MeSh-FrNA13} M. Megrelishvili and M. Shlossberg,
\textit{Free non-archimedean topological groups,}
Comm. Math. Univ. Carolinae 54:2, 2013, 273--312. 

\bibitem{MeShNAtransp} M. Megrelishvili and M. Shlossberg,
\textit{Non-Archimedean Transportation Problems and Kantorovich
Ultra-Norms}, p-Adic Numbers, Ultrametric Analysis and Applications, 2016, Vol. 8, No. 2, pp. 125--148. 

\bibitem{Pe-nbook}
V. Pestov, \textit{Dynamics of Infinite-dimensional Groups. The
	Ramsey-Dvoretzky-Milman Phenomenon,} University Lecture Series,
vol. 40 AMS, Providence, 2006.

\bibitem{Pest-Smirnov}
V. Pestov, 
\textit{A topological transformation group without non-trivial equivariant compactifications}, 
Advances in Math. 311 (2017) 1--17. 

\bibitem{RD}
W. Roelcke and S. Dierolf, {\em Uniform Structures on Topological
	Groups and Their Quotients\/}, McGraw-Hill, 1981.


\bibitem{Str} 
M. Stroppel, \textit{Locally Compact Groups}, EMS Textbooks in math. 2006. 


\bibitem{Usp-86} V.V. Uspenskij, \emph{A universal topological group with a
	countable basis}, Funct. Anal. Appl., \textbf{20} (1986),
160--161.


\bibitem{Vr-can75} J. de Vries, \emph{Can every Tychonoff $G$-space equivariantly
	be embedded in a compact Hausdorff $G$-space?}, Math. Centrum 36,
Amsterdam, Afd. Zuiver Wisk., 1975.



\bibitem{Vr-Embed77} J. de Vries, \emph{Equivariant embeddings of
	$G$-spaces}, in: J. Novak (ed.), General Topology and its
Relations to Modern Analysis and Algebra IV, Part B, Prague, 1977,
485--493.

\bibitem{Vr-loccom78} J. de Vries, \emph{On the existence of
	$G$-compactifications}, Bull. Acad. Polon. Sci. ser. Math. 26 (1978) 275--280.
%







\end{thebibliography}

\end{document}